\documentclass[11pt]{article}

\usepackage[margin=1in]{geometry}
\usepackage{amsmath,amssymb,amsthm,mathtools}
\usepackage{booktabs,longtable,threeparttable,array,tabularx}
\usepackage{enumitem}
\usepackage{siunitx}
\usepackage{xcolor}

\usepackage{hyperref}
\hypersetup{
  colorlinks=true,
  linkcolor=blue,
  citecolor=blue,
  urlcolor=blue
}
\usepackage[authoryear]{natbib}
\usepackage{microtype}
\usepackage[T1]{fontenc}
\usepackage{dsfont}
\numberwithin{equation}{section}

\newtheorem{theorem}{Theorem}[section]
\newtheorem{lemma}[theorem]{Lemma}
\newtheorem{proposition}[theorem]{Proposition}
\newtheorem{corollary}[theorem]{Corollary}
\theoremstyle{remark}
\newtheorem{remark}[theorem]{Remark}

\title{Pointwise Complexity for Gaussian Fields:\\Upper Envelopes, Algorithmic Lower Bounds, and Separation}
\author{Yunbei Xu\\National University of Singapore\\\texttt{yunbei@nus.edu.sg}}
\date{}

\begin{document}
\maketitle

\begin{abstract}
We prove a variance-aware pointwise majorizing-measure theorem for centered Gaussian processes. Classical generic chaining characterizes the scalar quantity \(\mathbb E\sup_{x\in T}X_x\); the theorem here gives a simultaneous high-probability envelope for the entire field. For an ambient prior \(\mu\), the envelope at \(x\) is governed by a {\it pointwise} Fernique--Talagrand functional \[\Phi_\mu(x)
:=
\int_0^{4\sigma(x)}
\sqrt{\log\frac{1}{\mu(B_d(x,\varepsilon))}}\,d\varepsilon,\] together with the corresponding Gaussian tail term. The theorem provides a reusable field-level refinement of classical generic chaining and a Gaussian-process counterpart of pointwise empirical-process bounds for deep neural networks.

We also record a Bayesian algorithmic lower envelope from the interactive Fano/data-processing principle. For a known prior \(\pi\), an observation channel, and a concrete estimator \(\widehat t(Y)\), the lower bound is expressed through the exact ghost small-ball mass \(\mathbb E_{Y\sim Q}\pi(B_d(\widehat t(Y),\Delta))\), rather than a worst-case covering number. In Gaussian location experiments, comparison decoders convert Bayes location error into lower bounds on decision-aligned Gaussian ranges.   We then construct an elementary weighted-basis example separating the usual Fano relaxation for a fixed prior, the Bayesian algorithmic lower envelope, the pointwise Gaussian envelope on the selected subatlas, and the full-class minimax risk/global Gaussian scale. Together, these results show that  algorithmic lower bounds provide local-geometric validations of pointwise complexity for fixed estimators in overparameterized ambient classes, precisely in regimes where classical minimax theory becomes either too coarse or oracle-dependent. 

This separation can also be recast in minimax language as penalty-range information relaxation, highlighting an important question of algorithmic robustness and information advantage for classical high-dimensional models and regularized algorithms.

The appendices further formalize analogies between renormalization group flows and deep neural networks, and between graph local-time and pointwise dimension, showing trajectory-aligned separations between pointwise and global complexity scales.
\end{abstract}

\tableofcontents

\section{Introduction and Main Results}

Classical generic chaining controls the scalar quantity \(\mathbb E\sup_{x\in T}X_x\) for a Gaussian process \(\{X_x\}_{x\in T}\). For many field-level questions, however, this scalar is not the first object one wants to estimate. In the graph setting, the second Ray--Knight theorem identifies a local-time field at an inverse-local-time stopping rule with a shifted square of a Gaussian free field. In constructive and lattice field theory, one similarly studies local observables before passing to global maxima or continuum limits. This suggests that the Gaussian input should itself be pointwise: one should first prove a simultaneous field-level envelope and only afterwards take suprema or threshold events.

A further motivation comes from recent work on pointwise generalization for deep neural networks \citep{li-xu-pointwise-generalization-dnn}. There the central object is not a class-wide supremum but a hypothesis-dependent finite-scale complexity. The proof is structural: an exact non-perturbative telescoping expansion exposes learned feature Gram matrices, and a hierarchical local-chart/global-atlas prior converts the resulting data-dependent active subspaces into a data-independent bound. The present paper develops the Gaussian-process counterpart of that viewpoint. Theorem~\ref{thm:pointwise-gaussian} is intended as a convenient reference whenever one needs a field-level refinement of classical generic chaining, and as the Gaussian-process counterpart of pointwise empirical-process bounds for deep neural networks.

We also record a single-radius Bayesian lower-envelope counterpart, derived from the data-processing principle of \citet{chen-foster-han-qian-rakhlin-xu-2024}, for a specified prior, observation channel, and estimator. In Gaussian location experiments, comparison decoders turn Bayes location error into an algorithmic lower bound on a decision-aligned Gaussian range. Conversely, a simultaneous pointwise Gaussian envelope yields an algorithmic upper bound on the same decoding risk. Ordinary nearest-neighbor decoding is the simplest instance; the separating example in Theorem~\ref{thm:decision-aligned-separation} uses an unrestricted norm-regularized nearest-neighbor rule over the full ambient class, reflecting the inductive-bias/interpolation principle that good empirical fit must be paired with a locality or complexity bias \citep{cover-hart-1967,belkin-hsu-mitra-2018}. We then construct an elementary weighted-basis ``hub--cloud'' example in which the fixed-prior Fano relaxation is vacuous, the algorithmic ghost-mass lower envelope is sharp for the chosen Bayes problem and for the pointwise upper envelope on the estimator-selected subatlas, while the full-class minimax risk and the global Gaussian supremum scale are much larger.

These results clarify why a Bayesian algorithmic lower bound can provide the right validation for a pointwise upper envelope under specified priors, even in overparameterized ambient classes where classical minimax or Bayes-risk criteria may become too large, prior-dependent, or oracle-dependent. In such regimes, there may be no ready-made notion of optimality that is both intrinsic and decision-aligned. The pointwise and algorithmic viewpoint developed here therefore motivates criteria localized to the algorithm or estimator and the reference law, paralleling probabilistic refinements of worst-case analysis in theoretical computer science \citep{yao1977probabilistic,levin1986average,spielman-teng-2004}.\footnote{Here ``algorithmic'' refers to bounds that depend on a concrete estimator and reference law. This usage is distinct from classical algorithmic information theory, which is traditionally formulated for finite strings or countable description spaces. For related notions of pointwise/local dimension and their connections to algorithmic and effective statistical dimension, see \citet{li-xu-pointwise-generalization-dnn,lutz2016note}.}

We further recast the separation in minimax terms, as a novel penalty-range information-relaxation reformulation. The penalty-range formulation also emphasizes an algorithmic-robustness question that is meaningful in traditional high-dimensional models, not only in overparameterized ones: how much structural information must be revealed as an oracle model class, and how much can instead be replaced by a stable tuning range for the estimator actually used \citep{tibshirani1996lasso,bickel-ritov-tsybakov-2009,belloni-chernozhukov-wang-2011,negahban-ravikumar-wainwright-yu-2012}?

The appendices formalize two structural analogies: one between exponential-family renormalization group (RG) flows and nonlinear feature learning in deep neural networks (DNN), and another between local time for continuous-time random walks on graphs and pointwise/local dimension in learning theory. The examples show explicit separations in which a trajectory- or target-aligned subatlas has a strictly smaller intrinsic scale than any global covering scale. These appendices are intended as initial results and guideposts for connections that have not yet been fully explored.

In summary, the paper connects field-level generic chaining with decision-aligned guarantees and demonstrates strong separations between pointwise and global complexity scales.
\subsection{Main pointwise Gaussian theorem}

Let \(\{X_x\}_{x\in T}\) be a centered Gaussian process on a finite index set \(T\). Fix an anchor \(x_0\in T\) such that \(X_{x_0}=0\) almost surely, and define
\[
d(x,y):=\bigl(\mathbb E|X_x-X_y|^2\bigr)^{1/2},
\qquad
\sigma(x):=(\mathbb E X_x^2)^{1/2}=d(x,x_0),
\qquad
R_\ast:=\sup_{x\in T}\sigma(x).
\]
For an ambient prior \(\mu\in\Delta(T)\), where \(\Delta(T)\) denotes the set of probability measures on \(T\), define the localized pointwise Fernique--Talagrand functional
\begin{equation}\label{eq:pointwise-phi-def}
\Phi_\mu(x)
:=
\int_0^{4\sigma(x)}
\sqrt{\log\frac{1}{\mu(B_d(x,\varepsilon))}}\,d\varepsilon,
\qquad
B_d(x,\varepsilon):=\{y\in T:d(x,y)\le \varepsilon\}.
\end{equation}
Here the upper limit is local: it is proportional to the standard deviation \(\sigma(x)\), rather than the global radius. The factor \(4\) is inessential; it absorbs the two constant-factor losses coming from anchoring and from the nearest-point pushforward used in the ambient-equivalence step. If one instead works with the non-local full-radius integral, then the upper limits \(R_\ast\) and \(2R_\ast\) are equivalent up to universal constants, since the entropy integrand is nonincreasing in \(\varepsilon\).

Write
\[
\Phi_{\mu,\ast}:=\sup_{x\in T}\Phi_\mu(x),
\qquad
\log_2^+(u):=\max\{\log_2 u,0\}.
\]
The degenerate case \(R_\ast=0\) is trivial, and the statements below are understood to be vacuous if \(\Phi_{\mu,\ast}=+\infty\).

For \(r_0>0\), \(s_0\in(0,R_\ast]\), and \(\delta\in(0,1)\), set
\begin{equation}\label{eq:peeling-multiplicity}
M_{r_0,s_0}
:=
\bigl(1+\lceil\log_2^+(\Phi_{\mu,\ast}/r_0)\rceil\bigr)
\bigl(1+\lceil\log_2(R_\ast/s_0)\rceil\bigr),
\end{equation}
and
\begin{equation}\label{eq:pointwise-envelope-shape}
\mathcal E_{\mu,r_0,s_0}(x;\delta)
:=
\max\{\Phi_\mu(x),r_0\}
+
\max\{\sigma(x),s_0\}
\sqrt{\log\!\left(\frac{eM_{r_0,s_0}}{\delta}\right)}.
\end{equation}
Thus all logarithmic losses from the two successive pointwise-peeling steps are isolated in \(M_{r_0,s_0}\).

We state the main result as a variance-aware pointwise Gaussian majorizing-measure theorem in high probability.

\begin{theorem}[Pointwise Gaussian majorizing-measure bound]
\label{thm:pointwise-gaussian}
Assume \(0<R_\ast<\infty\) and \(\Phi_{\mu,\ast}<\infty\). There exists an absolute constant \(C>0\) such that, for every ambient prior \(\mu\in\Delta(T)\), every \(r_0>0\), every \(s_0\in(0,R_\ast]\), and every \(\delta\in(0,1)\), with probability at least \(1-\delta\),
\begin{equation}\label{eq:pointwise-gaussian-rigorous}
\forall x\in T:\qquad
X_x\le C\,\mathcal E_{\mu,r_0,s_0}(x;\delta).
\end{equation}
Consequently, applying the same bound to \(X\) and \(-X\) and taking a union bound, with probability at least \(1-\delta\),
\begin{equation}\label{eq:absolute-pointwise-gaussian}
\forall x\in T:\qquad
|X_x|\le A_{\mu,r_0,s_0}(x;\delta),
\qquad
A_{\mu,r_0,s_0}(x;\delta):=
C\,\mathcal E_{\mu,r_0,s_0}(x;\delta/2).
\end{equation}
When \(r_0,s_0\) are fixed, we abbreviate \(A_{\mu,r_0,s_0}\) to \(A_\mu\).
\end{theorem}

Besides providing the Gaussian-process counterpart of the pointwise empirical-process bound in \citet[Theorem~1]{li-xu-pointwise-generalization-dnn}, the present result makes a key technical refinement: the global boundedness parameter for the metric is replaced by the local variance scale $4\sigma(x)$. Moreover, in the Gaussian-process setting we do not need the advanced symmetrization machinery or the mixed empirical-ghost metric used in \citet[Theorem~1]{li-xu-pointwise-generalization-dnn}.

Before presenting the in-expectation consequence, we record the corresponding anchored form of the classical generic-chaining scale. This is a localized reformulation of the usual Fernique--Talagrand theorem: since the cutoff $4\sigma(x)$ is not the standard full-radius presentation, the global radius $R_\ast$ must remain part of the anchored scale.
\begin{proposition}[Classical scale in anchored local form]
\label{prop:global-sharpness}
There exist absolute constants \(0<c<C<\infty\) such that
\begin{equation}\label{eq:gamma2-optimality}
c\,\Bigl(R_\ast+\inf_{\mu\in\Delta(T)}\Phi_{\mu,\ast}\Bigr)
\le
\mathbb E\sup_{x\in T}X_x
\le
C\,\Bigl(R_\ast+\inf_{\mu\in\Delta(T)}\Phi_{\mu,\ast}\Bigr).
\end{equation}
Moreover, since the process is centered and anchored,
\[
\mathbb E\sup_{x\in T}|X_x|
\asymp
\mathbb E\sup_{x\in T}X_x.
\]
Therefore
\begin{equation}\label{eq:absolute-global-optimality}
\mathbb E\sup_{x\in T}|X_x|
\asymp
R_\ast+\inf_{\mu\in\Delta(T)}\Phi_{\mu,\ast}.
\end{equation}
\end{proposition}

We now record the in-expectation consequence of Theorem~\ref{thm:pointwise-gaussian}. In particular, after optimizing over the ambient prior, the resulting global in-expectation bound is sharp up to universal constants. The variance-radius term is not an additional loss: it is part of the anchored generic-chaining scale in \eqref{eq:absolute-global-optimality}, while the peeling logarithms disappear by taking \(r_0\ge \Phi_{\mu,\ast}\) and \(s_0=R_\ast\).
\begin{corollary}[In-expectation consequence and sharp global scale]
\label{cor:pointwise-gaussian-expectation}
Under the assumptions of Theorem~\ref{thm:pointwise-gaussian}, there exists an absolute constant \(C>0\) such that
\begin{equation}\label{eq:normalized-expectation-envelope}
\mathbb E\left[
\sup_{x\in T}
\frac{
\bigl(|X_x|-C\max\{\Phi_\mu(x),r_0\}\bigr)_+
}{
\max\{\sigma(x),s_0\}
}
\right]
\le
C\sqrt{\log(eM_{r_0,s_0})}.
\end{equation}
Consequently, for every \(x\in T\),
\begin{equation}\label{eq:pointwise-expectation-envelope}
\mathbb E|X_x|
\le
C\max\{\Phi_\mu(x),r_0\}
+
C\max\{\sigma(x),s_0\}\sqrt{\log(eM_{r_0,s_0})}.
\end{equation}
At the global level, taking \(r_0\ge\Phi_{\mu,\ast}\) and \(s_0=R_\ast\) gives \(M_{r_0,s_0}=1\), hence
\begin{equation}\label{eq:global-expectation-clean}
\mathbb E\sup_{x\in T}|X_x|
\le
C(r_0+R_\ast).
\end{equation}
Letting \(r_0\downarrow \Phi_{\mu,\ast}\), then optimizing over \(\mu\), and using \eqref{eq:absolute-global-optimality}, we obtain the sharp equivalence
\begin{equation}\label{eq:optimized-expectation-optimal}
\mathbb E\sup_{x\in T}|X_x|
\asymp
R_\ast+\inf_{\mu\in\Delta(T)}\Phi_{\mu,\ast}.
\end{equation}
Thus the optimized expectation version is optimal up to absolute constants.
\end{corollary}

\paragraph{Remark (variance term versus chaining complexity).}
The term in Theorem~\ref{thm:pointwise-gaussian} involving \(\sigma(x)\)  is a variance-sensitive Gaussian tail term.  It is necessary for high-probability control, already for a single non-anchor point \(X_x\sim N(0,\sigma(x)^2)\), but it is not a replacement for the chaining complexity \(\Phi_\mu(x)\), which carries the local geometry of the index set.
For a fixed marginal,
\(\mathbb E|X_x|=\sqrt{\frac{2}{\pi}}\,\sigma(x)\),
so the extra factor \(\sqrt{\log(eM_{r_0,s_0})}\) in \eqref{eq:pointwise-expectation-envelope} is not intrinsic to \(\mathbb E|X_x|\). It is the price of extracting a simultaneous pointwise envelope by peeling first over the pointwise Fernique--Talagrand complexity and then over the local variance. At the global in-expectation level, this peeling factor disappears, yielding the sharp scale in \eqref{eq:optimized-expectation-optimal}.

\paragraph{Remark (pointwise high probability).} For Gaussian processes, subset-level in-expectation bounds transfer to high-probability bounds through Borell--Tsirelson concentration, with variance scale \(\sigma_H=\sup_{x\in H}(\mathbb E X_x^2)^{1/2}\) for a fixed subset \(H\). The genuinely pointwise statement is not obtained from this transfer alone: after the subset-homogeneous estimate is established, one first applies uniform pointwise peeling to localize the Fernique--Talagrand term \(\Phi_\mu(x)\), and then applies the same peeling mechanism a second time to improve the tail scale from the global radius to the local variance \(\sigma(x)\). Conversely, the resulting high-probability envelope implies the stated in-expectation consequence by integrating the tail.

\subsection{Main Bayesian algorithmic lower and upper bounds}\label{subsec:main-bayesian-lower}

The pointwise upper envelope is non-decision-theoretic, simultaneous, multiscale, and pathwise in the Gaussian field. A lower-bound analogue should therefore be stated with care. We first record a one-scale Bayesian and estimator-dependent primitive. Let \((T,d)\) be a finite metric space, let \(\pi\in\Delta(T)\) be a known prior, and let \(\{P_t:t\in T\}\) be an experiment on an observation space \(\mathcal Y\). For an estimator \(\widehat t:\mathcal Y\to T\), a reference law \(Q\in\Delta(\mathcal Y)\), and a scale \(\Delta>0\), define
\[
\rho_{\Delta,Q}(\widehat t)
:=\mathbb P_{t\sim\pi,\,Y\sim Q}\{d(t,\widehat t(Y))<\Delta\}
=\mathbb E_{Y\sim Q}\,\pi\!\left(B_d(\widehat t(Y),\Delta)\right),
\]
and
\[
\mathcal I_\pi(Q):=\mathbb E_{t\sim\pi}D_{\mathrm{KL}}(P_t\|Q).
\]
The following result is the Gaussian-process and KL specialization of the interactive Fano method of \citet[Theorem~2]{chen-foster-han-qian-rakhlin-xu-2024}. The original theorem provides a unifying
information-theoretic principle for proving lower bounds in learning and
decision-making problems.
\begin{proposition}[Bayesian algorithmic lower envelope]\label{prop:bayesian-algorithmic-lower}
For every estimator \(\widehat t\), every \(\Delta>0\), and every reference law \(Q\) with \(0<\rho_{\Delta,Q}(\widehat t)<1\),
\begin{equation}\label{eq:bayesian-algorithmic-lower-prob}
\mathbb P_{t\sim\pi,\,Y\sim P_t}\{d(t,\widehat t(Y))\ge\Delta\}
\ge
\left(
1-\frac{\mathcal I_{\pi}(Q)+\log 2}{\log(1/\rho_{\Delta,Q}(\widehat t))}
\right)_+ .
\end{equation}
Consequently,
\begin{equation}\label{eq:bayesian-algorithmic-lower-risk}
\mathbb E_{t\sim\pi,\,Y\sim P_t}d(t,\widehat t(Y))
\ge
\sup_{\substack{\Delta>0,\,Q:\\ 0<\rho_{\Delta,Q}(\widehat t)<1}}
\Delta
\left(
1-\frac{\mathcal I_{\pi}(Q)+\log 2}{\log(1/\rho_{\Delta,Q}(\widehat t))}
\right)_+ .
\end{equation}
\end{proposition}

The key point is that \(\rho_{\Delta,Q}(\widehat t)\) is not immediately replaced by the class-wide quantity
\[
q_\pi(\Delta):=\sup_{a\in T}\pi(B_d(a,\Delta)).
\]
Keeping the exact ghost mass \(\mathbb E_{Y\sim Q}\pi(B_d(\widehat t(Y),\Delta))\) makes the lower bound algorithmic: it is attached to the image of the actual estimator under ghost data. Replacing it by \(q_\pi(\Delta)\) gives the class-wide small-ball term used in classical Fano-type lower bounds \citep{zhang2006information} and in the fractional-covering specialization of \citet{chen-foster-han-qian-rakhlin-xu-2024}. In the Gaussian nearest-neighbor comparison below, this small-ball term must still be balanced against the information radius of the Gaussian channel.

We next present a lower-bound comparison for the Gaussian process itself,
rather than only for its induced distance.  The result is stated for a
comparison decoder.  This includes ordinary nearest neighbor and the
norm-regularized  rule used in the separating example in Theorem~\ref{thm:decision-aligned-separation}.

\begin{theorem}[Algorithmic lower bound for comparison decoders]
\label{thm:nn-comparison}
Let \(v:T\to\mathcal H\) be an embedding into a finite-dimensional Hilbert space and set
\(d(s,t)=\|v_s-v_t\|_{\mathcal H}\). Let
\[
Y=v_\Theta+\tau Z,
\qquad
\Theta\sim\pi,
\qquad
Z\sim N(0,I_{\mathcal H}),
\]
and let \(m\in\mathcal H\) be any reference center.  Set
\[
Q_{\tau,m}:=N(m,\tau^2I_{\mathcal H}),
\qquad
\mathcal I_{\pi,m}:=\frac{1}{2\tau^2}\int_T\|v_t-m\|^2\,\pi(dt).
\]
With \(\bar v_\pi:=\int_T v_t\,\pi(dt)\), the centered choice \(m=\bar v_\pi\) gives
\[
\mathcal I_{\pi,\bar v_\pi}=\frac{\mathsf V_\pi}{4\tau^2},
\qquad
\mathsf V_\pi:=\iint d(s,t)^2\pi(ds)\pi(dt).
\]

Let \(\widehat\Theta(Y)\) be an estimator with values in \(T\).
Assume that there is a penalty \(\Omega:T\to\mathbb R\) and nonnegative, integrable comparison errors
\(\varepsilon_{\rm opt}\) and \(\varepsilon_{\rm pen}\), measurable under the true joint law, such that
\begin{equation}\label{eq:comparison-decoder-condition}
\begin{gathered}
\|Y-v_{\widehat\Theta}\|^2+\Omega(\widehat\Theta)
\le
\|Y-v_\Theta\|^2+\Omega(\Theta)+\varepsilon_{\rm opt},\\
\Omega(\widehat\Theta)
\ge \Omega(\Theta)-\varepsilon_{\rm pen}.
\end{gathered}
\end{equation}
Let the full comparison error be denoted by
\[
\varepsilon_{\rm cmp}:=\varepsilon_{\rm opt}+\varepsilon_{\rm pen}.
\]
If, for some \(c\in(0,1)\) and \(\Delta>0\),
\begin{equation}\label{eq:bayesian-nn-condition}
\mathcal I_{\pi,m}+\log 2
\le
(1-c)\log\frac{1}{\rho_{\Delta,Q_{\tau,m}}(\widehat\Theta)},
\end{equation}
then
\begin{equation}\label{eq:bayesian-nn-risk-main}
\mathbb E_{\Theta,Y}d(\Theta,\widehat\Theta(Y))\ge c\Delta,
\end{equation}
and, for the Gaussian process \(X_t:=\langle Z,v_t\rangle\) built from the same Gaussian vector \(Z\),
\begin{equation}\label{eq:nn-to-gaussian-range-main}
\mathbb E_{\Theta,Z}\bigl[X_{\widehat\Theta}-X_\Theta\bigr]
\ge
\frac{c^2\Delta^2-\mathbb E_{\Theta,Z}\varepsilon_{\rm cmp}}{2\tau}.
\end{equation}
In particular, if \(\mathbb E\varepsilon_{\rm cmp}\le c^2\Delta^2/2\), then the right-hand side is at least \(c^2\Delta^2/(4\tau)\).  Thus the exact ghost mass gives an algorithmic lower bound on the Gaussian range sampled by the comparison decoder, up to the explicit comparison error.

Ordinary nearest neighbor corresponds to \(\Omega\equiv0\) and \(\varepsilon_{\rm opt}=\varepsilon_{\rm pen}=0\).  The unrestricted norm-regularized nearest-neighbor rule
\[
\widehat\Theta_\lambda(Y)\in
\arg\min_{a\in T}\bigl\{\|Y-v_a\|^2+\lambda\|v_a\|^2\bigr\}
\]
is also covered with zero error whenever the realized outputs have norm at least the norm of the true parameter, as in the separating example in Theorem~\ref{thm:decision-aligned-separation}. This is a canonical minimum-complexity inductive bias consistent with modern perspectives, rather than an explicit restriction of the hypothesis class.
\end{theorem}

We defer to Appendix~\ref{app:sudakov}
a route from the Bayesian algorithmic lower bound to a
Sudakov-type comparison via supremum relaxation, which consists of several technical results. Recent concurrent work \citep{zadik2026bayesian} shows that comparing nearest-neighbor and Bayes-optimal risks yields what is among the most succinct proofs to date of the majorizing-measure lower bound \citep{vanhandel2018chainingii,zadik2026bayesian}. Our aim here is instead conceptual: to show that the global Sudakov scale can be viewed as a coarse relaxation of a more localized algorithmic lower bound.

The same pathwise comparison underlying
Theorem~\ref{thm:nn-comparison}, together with the pointwise Gaussian
upper envelope, gives an algorithmic upper bound for Gaussian location decoding.
Thus, when the two estimates match, the Bayesian algorithmic lower bound and the
pointwise upper envelope validate the sharpness of one another.

\begin{corollary}[Algorithmic pointwise upper bound for estimator decoding]
\label{coro:estimation}
Work in the Gaussian location setup of Theorem~\ref{thm:nn-comparison}.
Thus \(T\) is embedded as \(\{v_t:t\in T\}\) in a Hilbert space,
\[
d(s,t)=\|v_s-v_t\|,
\qquad
Y=v_\Theta+\tau Z,
\qquad
X_t:=\langle Z,v_t\rangle ,
\]
where \(Z\) is standard Gaussian. Let \(\widehat\Theta=\widehat\Theta(Y)\) be a
decoder satisfying the approximate remainder comparison
\begin{equation}\label{eq:decoder-residual-comparison-for-upper}
\|Y-v_{\widehat\Theta(Y)}\|^2
\le
\|Y-v_\Theta\|^2+\varepsilon_{\rm up}
\end{equation}
under the joint law of \((\Theta,Y)\), for some nonnegative random variable \(\varepsilon_{\rm up}\).  For instance, \eqref{eq:comparison-decoder-condition} implies this condition with \(\varepsilon_{\rm up}=\varepsilon_{\rm cmp}\).

Suppose that a deterministic envelope \(A:T\to[0,\infty)\) satisfies
\begin{equation}\label{eq:generic-gaussian-envelope-for-upper}
\mathbb P_Z\bigl\{\forall t\in T:\ |X_t|\le A(t)\bigr\}\ge 1-\delta .
\end{equation}
Then, with probability at least \(1-\delta\) under the joint law of
\((\Theta,Y)\),
\begin{equation}\label{eq:algorithmic-upper-high-prob}
d(\Theta,\widehat\Theta(Y))^2
\le
2\tau\bigl(A(\Theta)+A(\widehat\Theta(Y))\bigr)+\varepsilon_{\rm up}.
\end{equation}
Consequently, if
\[
D_T:=\sup_{s,t\in T}d(s,t)<\infty,
\]
then
\begin{equation}\label{eq:algorithmic-upper-risk}
\mathbb E d(\Theta,\widehat\Theta(Y))
\le
\mathbb E\sqrt{2\tau\bigl(A(\Theta)+A(\widehat\Theta(Y))\bigr)+\varepsilon_{\rm up}}
+\delta D_T .
\end{equation}
In particular, after adjoining an anchor point \(t_0\) with \(v_{t_0}=0\) and
\(X_{t_0}=0\) if necessary, one may take
\[
A(t)=A_{\mu,r_0,s_0}(t;\delta)
\]
from \eqref{eq:absolute-pointwise-gaussian} in
Theorem~\ref{thm:pointwise-gaussian}, with the prior extended to
\(T\cup\{t_0\}\), for instance as \(\bar\mu=\frac12\delta_{t_0}+\frac12\mu\), which changes the pointwise functional on \(T\) only by universal constants. This yields a pointwise-complexity upper bound for the
estimation risk.
\end{corollary}

\paragraph{Why keep the Bayesian and  algorithmic form.}

The Bayesian algorithmic form is crucial because a full-class minimax benchmark can be the wrong validation for an overparameterized decision
problem, while a restricted minimax benchmark can be {\it oracle-dependent}. Full minimax over \(T\) can be too pessimistic, because it is governed by the worst point of the ambient class rather than by the geometry selected by the prior and the algorithm; maximizing over priors can similarly erase the structure that makes a particular estimator appropriate. On the other hand, a restricted minimax benchmark is useful only when the correct restricted class is  known in advance. In the separating example in Theorem~\ref{thm:decision-aligned-separation}, the relevant class is the low-norm hub subatlas \(H_M\), but identifying \(H_M\) as the right estimator-selected subatlas is precisely the decision-dependent information revealed by the ghost image \(Y\mapsto \widehat t(Y)\). Along the sequence \(N=M^K\) with \(K\to\infty\), the norm-regularized decoder uses a fixed numerical penalty \(\lambda\), independent of the prior $\pi$ and the problem parameters \(M,N,R,S,\tau\), and has Bayes risk \(\asymp R\), whereas the full-class minimax risk is \(\asymp S\), with
\(
\frac{S}{R}\asymp \sqrt{\frac{\log N}{\log M}}=\sqrt K\to\infty
\).
Thus the full-class benchmark misses the decision-aligned scale, while the restricted benchmark recovers it only after an  oracle has selected \(H_M\).

The same issue is compounded in deep networks: the effective  subatlas is rarely a fixed combinatorial class and may depend on the architecture, data distribution, learned representation, and compressed feature spectrum \citep{li-xu-pointwise-generalization-dnn}. Moreover, computationally efficient procedures may also operate through relaxations or implicit regularization rather than by solving an exact restricted-class problem, as in sparse regression. The algorithmic ghost mass \(\rho_{\Delta,Q}(\widehat t)\) is designed for this regime: it evaluates the estimator actually used, under the intended prior geometry, and therefore can validate a pointwise upper envelope without requiring a separate restricted class to be specified in advance.

The Bayesian {\it algorithmic} lower bounds
in Proposition~\ref{prop:bayesian-algorithmic-lower} and Theorem~\ref{thm:nn-comparison} should also not be read as ordinary Bayesian lower bounds whose
purpose is to optimize over all estimators given a specific prior. The prior is an averaging and
comparison device, while the validated object is the performance of a concrete estimator that is of
practical interest regardless of knowledge of the prior. This distinction is important because the
usual Fano relaxation may be vacuous for a fixed nonuniform prior even though the estimator of
interest remains hard on the subgeometry selected by the prior and the ghost law. Such vacuity does not mean that the Bayesian decision problem is easy; it means
only that the particular worst-case small-ball relaxation
\(\sup_a \pi(B_d(a,\Delta))\) has discarded the algorithmic geometry of
the estimator.
In the separating example of Theorem~\ref{thm:decision-aligned-separation}, the Bayes-optimal risk under the displayed prior is also $\asymp R$. Thus the class-wide Fano relaxation is vacuous only as a validation; it does not mean that the Bayes problem itself is easy. The same rotationally symmetric class construction can accommodate an
arbitrarily large finite family of priors that exhibit the separation and have markedly different local mass profiles: the Bayes-optimal oracle varies substantially across these priors, whereas the unrestricted estimator remains unchanged.

\paragraph{Relationship to the pointwise upper theorem and existing lower bounds.}
The Bayesian algorithmic lower bound also clarifies the relation to the pointwise upper theorem through Corollary~\ref{coro:estimation}. The upper envelope is simultaneous, multiscale, and high-probability: it is pathwise in the Gaussian field and integrates local prior masses across scales. By contrast, the lower bound is Bayesian, one-scale, and estimator-dependent: it averages over \(\Theta\sim\pi\) and \(Y\sim P_\Theta\) for a fixed estimator, while the additional ghost/reference law \(Y\sim Q\) enters through the quantile term \(\rho_{\Delta,Q}(\widehat t)\). This parallels the information-theoretic upper--lower duality emphasized by \citet{zhang2006information}, but with the additional point that the lower bound is a data-processing statement valid for every estimator. 

DEC-type lower bounds \citep{foster-kakade-qian-rakhlin-2021,foster-golowich-han-2023} provide an independent motivation for incorporating algorithmic dependence into the fundamental limits of online regret, exploration, and interactive decision making. This perspective further motivates the Bayesian algorithmic lower bound of \citet{chen-foster-han-qian-rakhlin-xu-2024} and related variants \citep{xu-zeevi-air-2025}. The present algorithmic lower-bound perspective isolates a distinct but
complementary principle: a prior-mass lower-bound analogue of pointwise
complexity, in the spirit of the pointwise generalization framework of
\citet{li-xu-pointwise-generalization-dnn} for deep representation
learning.

\subsection{Separating example construction and pointwise validation}\label{subsec:seperation-main}

The next theorem is the formal version of the separating example proved in Section~\ref{sec:hub-cloud-algorithmic-ghost}.  It shows that, in an overparameterized ambient class, a fixed-prior worst-case Fano relaxation, a Bayesian algorithmic lower envelope, a pointwise Gaussian upper envelope, and the full minimax/global Gaussian scale can all have different meanings and different orders.  The estimator is unrestricted over the full ambient class; the relevant subatlas is selected by a norm-regularized inductive bias and by the ghost law, rather than imposed as a constraint.

\begin{theorem}[Decision-aligned separation and pointwise validation]
\label{thm:decision-aligned-separation}
There exist universal constants \(0<c<C<\infty\) with the following property.  Let \(M,N\ge2\) satisfy \(\log N\ge 1024\log M\).  Let
\[
 e_0,e_1,\ldots,e_M,f_1,\ldots,f_N
\]
be orthonormal vectors in \(\mathbb R^{M+N+1}\).  Fix \(R>0\), define
\[
\tau:=\frac{R}{\sqrt{\log M}},
\qquad
S:=\frac18\tau\sqrt{\log N}
      =\frac{R}{8}\sqrt{\frac{\log N}{\log M}},
\]
and set
\[
H_M:=\{Re_0,Re_1,\ldots,Re_M\},
\qquad
C_N:=\{Sf_1,\ldots,Sf_N\},
\qquad
T_{M,N}:=H_M\cup C_N .
\]
Let \(d(s,t)=\|s-t\|_2\), let \(\Delta=R/3\), and let the prior be
\[
\pi(Re_0)=\frac12,
\qquad
\pi(Re_i)=\frac1{2M},\quad i=1,\ldots,M,
\qquad
\pi(Sf_j)=0,
\quad j=1,\ldots,N .
\]
Consider the Gaussian location experiment
\[
Y=\Theta+\tau Z,
\qquad
\Theta\sim\pi,
\qquad
Z\sim N(0,I_{M+N+1}),
\]
and the unrestricted norm-regularized decoder
\begin{equation}\label{eq:norm-regularized-decoder-main}
\widehat\Theta_\lambda(Y)
\in
\arg\min_{a\in T_{M,N}}\{\|Y-a\|^2+\lambda\|a\|^2\},
\qquad
\lambda=64 .
\end{equation}
Let \(Q_\tau=N(0,\tau^2I_{M+N+1})\).  Then, for all sufficiently large \(M,N\), the following hold.

\begin{enumerate}[label=\textnormal{(\roman*)},leftmargin=2.2em]
\item \emph{Fixed-prior Fano relaxation is vacuous.}  Since every \(\Delta\)-ball in \(T_{M,N}\) is a singleton,
\[
q_\pi(\Delta):=\sup_{a\in T_{M,N}}\pi(B_d(a,\Delta))=\frac12 .
\]
Moreover the usual data-processing lower envelope obtained by replacing the exact ghost mass by \(q_\pi(\Delta)\) gives zero at the reference law \(Q_\tau\).

\item \emph{Exact algorithmic ghost mass is small.}  The estimator-dependent ghost mass satisfies
\[
\rho_{\Delta,Q_\tau}(\widehat\Theta_\lambda)
=
\mathbb E_{Y\sim Q_\tau}\pi(B_d(\widehat\Theta_\lambda(Y),\Delta))
\le \frac{C}{M} .
\]
Consequently the exact ghost entropy is of order \(\log M\), whereas the worst-case Fano small-ball entropy for the same prior is only \(\log2\).

\item \emph{Sharp Bayes risk for the actual unrestricted estimator.}
\[
 cR
 \le
 \mathbb E_{\Theta,Y}d(\Theta,\widehat\Theta_\lambda(Y))
 \le
 CR .
\]
The lower bound follows from Proposition~\ref{prop:bayesian-algorithmic-lower} using the exact ghost mass; the upper bound is the actual Bayes risk of the same unrestricted estimator. 

Moreover, the Bayes-optimal risk for the displayed prior is also $\asymp R$. Thus the class-wide Fano relaxation below is vacuous only as a validation; it does not reflect an easy Bayes problem. The same rotationally symmetric class construction can accommodate an arbitrarily large finite family of priors exhibiting the separation, with markedly different local mass profiles. Across this family, the Bayes-optimal oracle may vary substantially, while the unrestricted estimator remains unchanged.

\item \emph{Full-class minimax risk is much larger.}  If
\[
\mathfrak R_{\rm mm}(T_{M,N},\tau)
:=
\inf_{\widehat\theta}
\sup_{t\in T_{M,N}}
\mathbb E_t d(t,\widehat\theta(Y)),
\]
where the infimum is over all measurable estimators with values in the ambient Euclidean space, then
\[
 cS\le \mathfrak R_{\rm mm}(T_{M,N},\tau)\le CS .
\]
Thus, whenever \(\log N/\log M\to\infty\), the full-class minimax risk is asymptotically larger than the Bayes risk under \(\pi\).

\item \emph{Algorithmically induced Gaussian validation matches the selected subatlas.}
For the Gaussian process
\[
  X_t := \langle Z,t\rangle,\qquad t\in T_{M,N},
\]
where \(Z\) is the same standard Gaussian vector as in the location experiment
\(Y=\Theta+\tau Z\), define
\[
  \mathcal G_{\rm alg}
  :=
  \mathbb E_{\Theta,Z}
  \bigl[
    X_{\widehat\Theta_\lambda(\Theta+\tau Z)}-X_\Theta
  \bigr].
\]
Then
\[
  \mathcal G_{\rm alg}
  \asymp
  \mathbb E\sup_{s,u\in H_M}(X_s-X_u)
  \asymp
  R\sqrt{\log M}.
\]
Moreover, under the ghost/reference law \(Y\sim Q_\tau=N(0,\tau^2I_{M+N+1})\),
\[
  Q_\tau\{\widehat\Theta_\lambda(Y)\in C_N\}\le N^{-2}
\]
for all sufficiently large $N$, and hence the ghost-selected subatlas is \(H_M\) up to negligible probability.  Conditional on
\(\widehat\Theta_\lambda(Y)\in H_M\), the ghost selection law is uniform on \(H_M\).  Therefore,
after adjoining the zero anchor, the pointwise Gaussian upper theorem applied to the restricted
process on \(H_M\) gives the matching high-probability envelope
\[
  \mathbb P\!\left(
    \forall t\in H_M:\ |X_t|\le C R\sqrt{\log\frac{M}{\delta}}
  \right)
  \ge 1-\delta .
\]

\item \emph{Global Gaussian scale can be arbitrarily larger.}
\[
\mathbb E\sup_{s,u\in T_{M,N}}(X_s-X_u)
\ge
cS\sqrt{\log N}
=
\frac{c}{8}R\frac{\log N}{\sqrt{\log M}} .
\]
Thus the global Gaussian scale exceeds the decision-aligned scale \(R\sqrt{\log M}\) by a factor of order \(\log N/\log M\), which can be made arbitrarily large by taking \(N=M^K\) with \(K\to\infty\).
\end{enumerate}
Consequently the same explicit problem sequence separates
\begin{align*}
&\text{fixed-prior Fano relaxation}\\
\quad <\!\!< \quad
&\text{algorithmic ghost-mass lower envelope}\\
\quad \asymp \quad
&\text{pointwise upper envelope on the selected subatlas}\\
\quad <\!\!< \quad
&\text{full-class minimax/global Gaussian scale}.
\end{align*}
\end{theorem}

\subsection{Minimax reformulation, information relaxation, and algorithmic robustness}
The next theorem isolates the minimax language behind the same example in a form that reveals algorithmic tuning information, but not an oracle geometric restriction.  It is motivated by the information-relaxation and adaptivity-gap viewpoint of \citet{maiti-xu-jamieson-2026}: in their online problem, adaptive sampling can uncover structural value-range information that a nonadaptive design does not possess.  The offline analogue considered here deliberately reveals a different object.  Rather than allowing the learner to replace the ambient atlas by the hub subatlas after being told a norm range for \(\theta^*\), we reveal only a penalty range for which the full-atlas norm-regularized nearest-neighbor rule is stable.  The resulting value is minimax after the penalty range has been revealed, but it remains algorithm-dependent because the feasible estimators are the members of the fixed full-atlas regularized-nearest-neighbor family.  This is why the statement is a reformulation of the algorithmic separation, not an intrinsic restricted-minimax theorem.

\begin{theorem}[Penalty-range information-relaxation reformulation]
\label{thm:penalty-range-information-relaxation}
Consider the construction of Theorem~\ref{thm:decision-aligned-separation} along any sequence with
\[
M\to\infty,
\qquad
N\to\infty,
\qquad
\frac{\log N}{\log M}\to\infty .
\]
Equivalently, in the notation of that theorem, \(S/R=(1/8)\sqrt{\log N/\log M}\to\infty\).  For every \(\lambda\ge0\), let the full-atlas norm-regularized nearest-neighbor rule be
\begin{equation}\label{eq:lambda-family-penalty-revelation}
\widehat\Theta_\lambda(Y)
\in
\arg\min_{a\in T_{M,N}}\{\|Y-a\|_2^2+\lambda\|a\|_2^2\},
\end{equation}
with an arbitrary deterministic tie-breaking rule.  Define the critical penalty
\[
\lambda_{\rm c}:=16\sqrt2-1,
\]
and, for a fixed \(\eta\in(0,\lambda_{\rm c})\), define the separated penalty ranges
\[
\Lambda_{\rm opt}(\eta):=[\lambda_{\rm c}+\eta,\infty),
\qquad
\Lambda_{\rm under}(\eta):=[0,\lambda_{\rm c}-\eta].
\]
For a nonempty penalty range \(\Lambda\subset[0,\infty)\), define the penalty-revealed algorithmic minimax value over the low-loss hub regime by
\begin{equation}\label{eq:penalty-revealed-minimax-value}
\mathfrak R_{\rm RNN}^{H}(\Lambda)
:=
\inf_{\lambda\in\Lambda}
\sup_{\theta\in H_M}
\mathbb E_\theta d\bigl(\theta,\widehat\Theta_\lambda(Y)\bigr).
\end{equation}
The estimator in \eqref{eq:penalty-revealed-minimax-value} is not allowed to restrict the feasible set to \(H_M\): for every \(\lambda\), it is exactly the full-atlas rule \eqref{eq:lambda-family-penalty-revelation} over \(T_{M,N}\).  The revealed side information is only that the penalty lies in \(\Lambda\).

Then the penalty transition is sharp up to the arbitrary margin \(\eta\).  There is a constant \(a_\eta>0\), depending only on \(\eta\), such that, for all sufficiently large \(M,N\),
\begin{equation}\label{eq:opt-penalty-cloud-suppression}
\sup_{\lambda\in\Lambda_{\rm opt}(\eta)}
\left[
Q_\tau\{\widehat\Theta_\lambda(Y)\in C_N\}
+
\sup_{\theta\in H_M}P_\theta\{\widehat\Theta_\lambda(Y)\in C_N\}
\right]
\le N^{-a_\eta},
\end{equation}
where \(Q_\tau=N(0,\tau^2I_{M+N+1})\) and \(P_\theta=N(\theta,
\tau^2I_{M+N+1})\).  Consequently, for every \(\lambda\in\Lambda_{\rm opt}(\eta)\), the ghost mass at \(\Delta=R/3\) satisfies
\begin{equation}\label{eq:opt-penalty-ghost-mass}
\rho_{\Delta,Q_\tau}(\widehat\Theta_\lambda)
=
\mathbb E_{Y\sim Q_\tau}\pi\bigl(B_d(\widehat\Theta_\lambda(Y),\Delta)\bigr)
\le \frac{C}{M}.
\end{equation}
In the opposite, under-regularized range,
\begin{equation}\label{eq:under-penalty-cloud-selection}
\inf_{\lambda\in\Lambda_{\rm under}(\eta)}
\inf_{\theta\in H_M}
P_\theta\{\widehat\Theta_\lambda(Y)\in C_N\}
\longrightarrow 1 .
\end{equation}
Therefore, for constants \(0<c_\eta<C_\eta<\infty\), depending only on the fixed margin \(\eta\),
\begin{equation}\label{eq:penalty-minimax-gap}
c_\eta R
\le
\mathfrak R_{\rm RNN}^{H}(\Lambda_{\rm opt}(\eta))
\le
C_\eta R,
\qquad
c_\eta S
\le
\mathfrak R_{\rm RNN}^{H}(\Lambda_{\rm under}(\eta))
\le
C_\eta S .
\end{equation}
The same ambient problem still has full-class minimax risk
\begin{equation}\label{eq:penalty-full-minimax}
\mathfrak R_{\rm mm}^{\rm full}(T_{M,N},\tau)
:=
\inf_{\widehat\theta}
\sup_{\theta\in T_{M,N}}
\mathbb E_\theta d(\theta,\widehat\theta(Y))
\asymp S,
\end{equation}
where the infimum is over all measurable estimators with values in the ambient Euclidean space.  Hence
\[
\frac{\mathfrak R_{\rm RNN}^{H}(\Lambda_{\rm under}(\eta))}
     {\mathfrak R_{\rm RNN}^{H}(\Lambda_{\rm opt}(\eta))}
\asymp
\frac{\mathfrak R_{\rm mm}^{\rm full}(T_{M,N},\tau)}
     {\mathfrak R_{\rm RNN}^{H}(\Lambda_{\rm opt}(\eta))}
\asymp
\frac{S}{R}
=
\frac18\sqrt{\frac{\log N}{\log M}}.
\]
For the associated Gaussian field \(X_t=\langle Z,t\rangle\), every \(\lambda\in\Lambda_{\rm opt}(\eta)\) satisfies the approximate comparison condition \eqref{eq:comparison-decoder-condition} with \(\Omega(t)=\lambda\|t\|_2^2\) and \(\varepsilon_{\rm cmp}=0\).  Thus Theorem~\ref{thm:nn-comparison}, together with \eqref{eq:opt-penalty-ghost-mass}, gives
\begin{equation}\label{eq:penalty-revealed-gaussian-alg}
\mathcal G_{\rm alg}(\lambda)
:=
\mathbb E_{\Theta,Z}
\bigl[X_{\widehat\Theta_\lambda(\Theta+\tau Z)}-X_\Theta\bigr]
\asymp
R\sqrt{\log M},
\qquad
\lambda\in\Lambda_{\rm opt}(\eta),
\end{equation}
where \(\Theta\sim\pi\) is the prior of Theorem~\ref{thm:decision-aligned-separation}.  This matches the hub Gaussian range
\[
\Gamma_H:=\mathbb E\sup_{s,u\in H_M}(X_s-X_u)\asymp R\sqrt{\log M},
\]
whereas the full ambient Gaussian range remains
\[
\Gamma_{\rm full}:=
\mathbb E\sup_{s,u\in T_{M,N}}(X_s-X_u)
\asymp
S\sqrt{\log N}
=
\frac18R\frac{\log N}{\sqrt{\log M}}.
\]
Along \(N=M^K\) with \(K\to\infty\), the penalty-revealed minimax-risk gap is \(\asymp\sqrt K\), and the Gaussian-range gap is \(\asymp K\).
\end{theorem}
\begin{remark}[Penalty-range information and algorithmic robustness]
The algorithmic-minimax separation central to the paper can therefore be reformulated as penalty-range information relaxation: the learner is told a stable range for the regularization parameter in a fixed full-atlas algorithmic family, not a norm shell on which arbitrary subatlas ERM could be run.  This distinction is essential.  Revealing the hub norm shell would change the statistical problem by supplying an oracle restricted class; revealing \(\Lambda_{\rm opt}(\eta)\) instead supplies only the kind of tuning information that the chosen algorithm can use.  The formulation is minimax, since \eqref{eq:penalty-revealed-minimax-value} still takes a worst-case risk over the low-loss hub regime after the side information is supplied; it is also algorithm-dependent, since the learner may optimize over \(\lambda\in\Lambda\) but may not leave the full-atlas regularized-nearest-neighbor family. Borrowing the related but distinct terminology of information relaxation from stochastic dynamic programs \citep{brown-smith-sun-2010}, which is particularly suitable here, we refer to this reformulation as \emph{penalty-range information relaxation}.

This formulation suggests a broader inquiry into algorithmic information gaps.  The offline penalty-revelation gap proved here and the online adaptivity gap of \citet{maiti-xu-jamieson-2026} are two complementary examples: in both cases, the relevant comparison depends not only on the statistical model class, but also on which structural information is made available to a particular algorithmic family.  This is why the present theorem is a minimax reformulation, rather than a replacement for ordinary minimax theory: it separates information relaxation for a procedure from oracle revelation of a smaller model class.

The viewpoint is of independent interest beyond the overparameterized regime emphasized in the example.  Classical penalized model selection already uses penalties to select among candidate dimensions without revealing the true model in advance \citep{barron-birge-massart-1999,birge-massart-2007}.  Sparse linear regression provides an especially familiar high-dimensional case.  The Lasso \citep{tibshirani1996lasso} is not merely a computational relaxation of best-subset search; under restricted-eigenvalue or compatibility conditions, Lasso-type estimators achieve oracle inequalities with a penalty of order \(\sigma\sqrt{(\log p)/n}\), so the tuning rule does not require the exact sparsity level \(s\), even though the resulting rate depends on \(s\) \citep{bickel-ritov-tsybakov-2009,buhlmann-van-de-geer-2011}.  Square-root Lasso makes this robustness even more explicit by using a pivotal penalty that does not require prior knowledge of the noise standard deviation \citep{belloni-chernozhukov-wang-2011}.  More broadly, decomposable regularized \(M\)-estimators use regularizer geometry and restricted strong convexity to obtain rates across sparse, group-sparse, and low-rank models without being handed the exact active subatlas \citep{negahban-ravikumar-wainwright-yu-2012,hastie-tibshirani-wainwright-2015,wainwright-2019-highdim}.  Minimax analyses of sparse regression make the same distinction visible: computationally tractable \(\ell_1\) relaxations can attain sparse rates under suitable design assumptions, whereas oracle constrained least squares is formulated by first specifying the \(\ell_0\)-class \citep{raskutti-wainwright-yu-2011}.  In this algorithm-dependent information sense, the Lasso's appeal is not only computational but also statistical and information-theoretic: it needs a stable penalty range rather than exact model-size information.  The boundary between traditional high-dimensional statistics and modern overparameterized models is therefore methodological rather than sharp; both ask when a stable algorithmic bias can replace explicit oracle information.
\end{remark}

\subsection{Organization}

Section~\ref{sec:gaussian-proof} proves the pointwise Gaussian majorizing-measure theorem. Section~\ref{sec:proof-bayesian-lower} proves the Bayesian algorithmic lower and upper bounds. Section~\ref{sec:hub-cloud-algorithmic-ghost} proves the separating example and gives its pointwise validation. Section~\ref{sec:minimax} develops the corresponding penalty-range minimax information-relaxation reformulation. Section~\ref{sec:summary} closes the main text. Appendix~\ref{sec:finite-rg} contains the finite-cutoff renormalization material. Appendix~\ref{sec:graph-background} contains the graph local-time application and cover-time comparison. Appendix~\ref{app:sudakov} records technical variants and Sudakov-type consequences. 

The appendices are intentionally separated from the main Gaussian and Bayesian results.  They record structural analogies and first finite examples, rather than claiming a complete renormalization group theory or a complete theory of local time, cover time, and blanket time.

\section{Proof of Pointwise Gaussian Majorizing-Measure Theorem}\label{sec:gaussian-proof}
The proof follows a subset-homogeneity and peeling strategy. First, for each fixed subset \(H\subseteq T\), we combine an anchored form of Fernique's majorizing-measure upper bound \citep{fernique1975regularite,talagrand1987regularity}, the ambient equivalence principle of pointwise dimension, and Borell--Tsirelson concentration. Second, we apply the one-parameter uniform pointwise peeling lemma twice: first to localize the Fernique--Talagrand integral \(\Phi_\mu(x)\), and then to localize the variance scale from \(R_\ast\) to \(\sigma(x)\). This avoids introducing a separate two-dimensional peeling corollary.

The following uniform pointwise peeling lemma was originally developed
for empirical processes in \cite{xu2025towards}. It identifies minimal
conditions for obtaining pointwise envelopes, up to negligible
double-logarithmic factors, without imposing unrealistic Bernstein-type
or rigid sub-root assumptions. Necessary and sufficient justifications
and discussions of data dependence are provided in
\cite{li-xu-pointwise-generalization-dnn}. Here we extend the lemma to
general indexed processes, including Gaussian processes.
\begin{lemma}[Uniform pointwise peeling for indexed processes]\label{lem:gaussian-peeling}
Let \(\{Z_x\}_{x\in T}\) be random variables, let \(d_0:T\to[0,R]\), and let \(\psi(r;\delta)\) be nondecreasing in \(r\) and nonincreasing in \(\delta\). Assume that for every fixed subset \(H\subseteq T\) and every \(\delta\in(0,1)\),
\begin{equation}\label{eq:subset-homogeneous-abstract}
\mathbb P\!\left(\sup_{x\in H} Z_x\le \sup_{x\in H}\psi(d_0(x);\delta)\right)\ge 1-\delta.
\end{equation}
Then for every \(r_0>0\), with probability at least \(1-\delta\), uniformly over all \(x\in T\),
\begin{equation}\label{eq:abstract-pointwise-peeling}
Z_x\le \psi\!\left(\max\{2d_0(x),r_0\};\frac{\delta}{1+\lceil\log_2^+(R/r_0)\rceil}\right),
\end{equation}
where \(\log_2^+(u)=\max\{\log_2 u,0\}\).
\end{lemma}

\begin{proof}
Let \(m:=\lceil\log_2^+(R/r_0)\rceil\). The sets
\[
H_0:=\{x:d_0(x)\le r_0\},
\qquad
H_j:=\{x:2^{j-1}r_0<d_0(x)\le 2^j r_0\}\quad (1\le j\le m)
\]
cover \(T\). Apply \eqref{eq:subset-homogeneous-abstract} to each \(H_j\) with confidence level \(\delta/(m+1)\), and take a union bound. On the resulting event, if \(x\in H_0\), then
\[
Z_x\le \psi\!\left(r_0;\frac{\delta}{m+1}\right)
\le \psi\!\left(\max\{2d_0(x),r_0\};\frac{\delta}{m+1}\right),
\]
and if \(x\in H_j\) for \(j\ge1\), then \(2^jr_0\le 2d_0(x)\), so
\[
Z_x\le \psi\!\left(2^jr_0;\frac{\delta}{m+1}\right)
\le \psi\!\left(\max\{2d_0(x),r_0\};\frac{\delta}{m+1}\right).
\]
This proves the claim.
\end{proof}

We also use the following localized version of the ambient equivalence principle of pointwise dimension from Lemma~7 of \cite{li-xu-pointwise-generalization-dnn}.

\begin{lemma}[Ambient equivalence for the localized integral]\label{lem:pd-equivalence}
Let \((T,d)\) be a finite metric space and let \(H\subseteq T\). Let \(p:T\to H\) be a nearest-point selector, so that \(d(y,p(y))=\min_{h\in H}d(y,h)\), and let \(\mu_H=p_\#\mu\) be the pushforward of an ambient prior \(\mu\in\Delta(T)\). Then, for every \(x\in H\) and every \(\varepsilon>0\),
\[
\mu_H(B_d(x,2\varepsilon))\ge \mu(B_d(x,\varepsilon)).
\]
Consequently, for the localized functional \(\Phi\) in \eqref{eq:pointwise-phi-def},
\begin{equation}\label{eq:ambient-equivalence-bound}
\Phi_{\mu_H}(x)\le 2\Phi_\mu(x),
\qquad x\in H.
\end{equation}
\end{lemma}

\begin{proof}
If \(y\in B_d(x,\varepsilon)\), then
\[
d(p(y),x)\le d(p(y),y)+d(y,x)\le 2d(y,x)\le 2\varepsilon,
\]
because \(x\in H\) and \(p(y)\) is a nearest point in \(H\). Hence \(p(B_d(x,\varepsilon))\subseteq B_d(x,2\varepsilon)\), which gives the ball-mass inequality. Therefore
\[
\Phi_{\mu_H}(x)
=\int_0^{4\sigma(x)}\sqrt{\log\frac1{\mu_H(B_d(x,u))}}\,du
\le
2\int_0^{2\sigma(x)}\sqrt{\log\frac1{\mu(B_d(x,\varepsilon))}}\,d\varepsilon
\le 2\Phi_\mu(x),
\]
where we used the change of variables \(u=2\varepsilon\).
\end{proof}

\begin{lemma}[Anchored subset majorizing-measure bound]\label{lem:anchored-subset-mm}
For every fixed subset \(H\subseteq T\), every ambient prior \(\mu\in\Delta(T)\), and
\[
\sigma_H:=\sup_{x\in H}\sigma(x),
\]
there is an absolute constant \(C>0\) such that
\begin{equation}\label{eq:anchored-subset-mm}
\mathbb E\sup_{x\in H}X_x
\le
C\sup_{x\in H}\Phi_\mu(x)+C\sigma_H.
\end{equation}
\end{lemma}

\begin{proof}
Let \(H_0:=H\cup\{x_0\}\), and let \(p:T\to H_0\) be a nearest-point selector. Define the probability measure
\[
\bar\mu_H:=\frac12\delta_{x_0}+\frac12 p_\#\mu
\]
on \(H_0\). Fernique's majorizing-measure upper bound \citep{fernique1975regularite,talagrand1987regularity}, applied to the restricted process on \(H_0\), gives
\[
\mathbb E\sup_{x\in H}X_x
\le
C\sup_{x\in H_0}
\int_0^{2\sigma_H}
\sqrt{\log\frac1{\bar\mu_H(B_d(x,u))}}\,du,
\]
since \(\operatorname{diam}(H_0,d)\le 2\sigma_H\). For \(x=x_0\), the integral is at most \(2\sigma_H\sqrt{\log2}\), because \(\bar\mu_H\{x_0\}\ge1/2\). For \(x\in H\), split the integral at \(4\sigma(x)\wedge 2\sigma_H\).  If \(4\sigma(x)>2\sigma_H\), the second interval below is empty.  On \([4\sigma(x),2\sigma_H]\), the ball \(B_d(x,u)\) contains \(x_0\), hence has \(\bar\mu_H\)-mass at least \(1/2\), so this part is bounded by \(C\sigma_H\). On \([0,4\sigma(x)]\), Lemma~\ref{lem:pd-equivalence} and the extra factor \(1/2\) in \(\bar\mu_H\) give a bound of order \(\Phi_\mu(x)+\sigma(x)\). Taking the supremum over \(x\in H\) proves \eqref{eq:anchored-subset-mm}.
\end{proof}

\begin{lemma}[Subset-homogeneous Gaussian estimate]\label{lem:subset-gaussian-estimate}
For every fixed subset \(H\subseteq T\), every ambient prior \(\mu\in\Delta(T)\), and every \(\delta\in(0,1)\),
\begin{equation}\label{eq:subset-gaussian-estimate}
\mathbb P\left(
\sup_{x\in H}X_x
\le
C\sup_{x\in H}\Phi_\mu(x)
+
C\sigma_H\sqrt{\log(e/\delta)}
\right)
\ge 1-\delta,
\end{equation}
where \(\sigma_H=\sup_{x\in H}\sigma(x)\).
\end{lemma}

\begin{proof}
By Lemma~\ref{lem:anchored-subset-mm},
\[
\mathbb E\sup_{x\in H}X_x\le C\sup_{x\in H}\Phi_\mu(x)+C\sigma_H.
\]
Realize \(X_x=\langle g,v_x\rangle\) as an isonormal process on a Hilbert space. The map \(g\mapsto\sup_{x\in H}\langle g,v_x\rangle\) is \(\sigma_H\)-Lipschitz. Borell--Tsirelson concentration \citep{borell-1975} therefore gives the desired inequality, after absorbing the expectation term \(C\sigma_H\) into \(C\sigma_H\sqrt{\log(e/\delta)}\).
\end{proof}

\begin{proof}[Proof of Theorem~\ref{thm:pointwise-gaussian}]
Fix an arbitrary subset \(H\subseteq T\). Applying Lemma~\ref{lem:gaussian-peeling} on the index set \(H\), with \(Z_x=X_x\), \(d_0(x)=\Phi_\mu(x)\), \(R=\Phi_{\mu,\ast}\), and using the subset estimate \eqref{eq:subset-gaussian-estimate}, yields the following: for every \(r_0>0\) and every \(\delta\in(0,1)\), with probability at least \(1-\delta\),
\begin{equation}\label{eq:first-peeling-output}
\forall x\in H:\qquad
X_x
\le
C\max\{\Phi_\mu(x),r_0\}
+
C\sigma_H
\sqrt{\log\!\left(\frac{eM_\Phi}{\delta}\right)},
\end{equation}
where
\[
M_\Phi:=1+\lceil\log_2^+(\Phi_{\mu,\ast}/r_0)\rceil.
\]
Since \(H\) was arbitrary, \eqref{eq:first-peeling-output} is now a subset-homogeneous estimate for the centered residual
\[
Y_x:=X_x-C\max\{\Phi_\mu(x),r_0\}
\]
with complexity \(\sigma(x)\). Applying Lemma~\ref{lem:gaussian-peeling} a second time, now with \(Z_x=Y_x\), \(d_0(x)=\sigma(x)\), \(R=R_\ast\), and
\[
\psi(s;\delta):=Cs\sqrt{\log\!\left(\frac{eM_\Phi}{\delta}\right)},
\]
gives, with probability at least \(1-\delta\),
\[
\forall x\in T:\qquad
Y_x
\le
C\max\{\sigma(x),s_0\}
\sqrt{\log\!\left(\frac{eM_\Phi M_\sigma}{\delta}\right)},
\]
where
\[
M_\sigma:=1+\lceil\log_2(R_\ast/s_0)\rceil.
\]
Combining the last display with the definition of \(Y_x\) proves \eqref{eq:pointwise-gaussian-rigorous}, since \(M_{r_0,s_0}=M_\Phi M_\sigma\). The absolute-value estimate \eqref{eq:absolute-pointwise-gaussian} follows by applying the same argument to \(-X\) and taking a union bound.
\end{proof}

\begin{proof}[Proof of Proposition~\ref{prop:global-sharpness}]
The upper bound follows from Lemma~\ref{lem:anchored-subset-mm} with \(H=T\), followed by optimization over \(\mu\). For the lower bound, choose \(x_\ast\) with \(\sigma(x_\ast)=R_\ast\). Since \(X_{x_0}=0\),
\[
\sup_{x\in T}X_x\ge (X_{x_\ast})_+,
\qquad
\mathbb E(X_{x_\ast})_+=\frac{R_\ast}{\sqrt{2\pi}}.
\]
It remains to compare the optimized local functional with the classical full-radius one.  Let
\[
\Phi_\mu^{\rm full}(x):=\int_0^{2R_\ast}\sqrt{\log\frac1{\mu(B_d(x,\varepsilon))}}\,d\varepsilon .
\]
The classical majorizing-measure theorem gives
\[
\inf_{\mu\in\Delta(T)}\sup_x \Phi_\mu^{\rm full}(x)
\lesssim \mathbb E\sup_{x\in T}X_x.
\]
Since \(\Phi_\mu(x)\le \Phi_\mu^{\rm full}(x)\) for every \(x\), we have
\[
\inf_\mu\Phi_{\mu,\ast}\le
\inf_\mu\sup_x\Phi_\mu^{\rm full}(x)
\lesssim \mathbb E\sup_{x\in T}X_x.
\]
Together with the radius lower bound this proves the left side of \eqref{eq:gamma2-optimality}. The absolute-value statement follows from
\[
\sup_x X_x\le \sup_x |X_x|\le \sup_x X_x+\sup_x(-X_x)
\]
and symmetry.
\end{proof}

\begin{proof}[Proof of Corollary~\ref{cor:pointwise-gaussian-expectation}]
Let
\[
Y:=
\sup_{x\in T}
\frac{
\bigl(|X_x|-C\max\{\Phi_\mu(x),r_0\}\bigr)_+
}{
\max\{\sigma(x),s_0\}
}.
\]
By the two-sided form of Theorem~\ref{thm:pointwise-gaussian}, after increasing \(C\) if necessary,
\[
\mathbb P\left(
Y>C\sqrt{\log\!\left(\frac{eM_{r_0,s_0}}{\delta}\right)}
\right)
\le \delta,
\qquad \delta\in(0,1).
\]
Integrating this tail gives
\[
\mathbb EY\le C\sqrt{\log(eM_{r_0,s_0})},
\]
which proves \eqref{eq:normalized-expectation-envelope}. The pointwise bound \eqref{eq:pointwise-expectation-envelope} follows by dropping the supremum, and the global bound \eqref{eq:global-expectation-clean} follows by taking \(r_0\ge\Phi_{\mu,\ast}\), \(s_0=R_\ast\), and then letting \(r_0\downarrow\Phi_{\mu,\ast}\). Optimizing over \(\mu\) and applying Proposition~\ref{prop:global-sharpness} proves \eqref{eq:optimized-expectation-optimal}.
\end{proof}

\section{Proof of Bayesian Algorithmic Lower and Upper Bounds}\label{sec:proof-bayesian-lower}
\begin{proof}[Proof of Proposition~\ref{prop:bayesian-algorithmic-lower}]
The result is the Gaussian-process and KL specialization of the interactive Fano method of \citet[Theorem~2]{chen-foster-han-qian-rakhlin-xu-2024}; we give a self-contained proof. Fix \(\widehat t\), \(Q\), and \(\Delta\), and let
\[
E:=\{(t,Y):d(t,\widehat t(Y))<\Delta\}.
\]
Under the true joint law \(P:=\int \pi(dt)P_t(dY)\), write \(p=P(E)\). Under the ghost law \(P_0:=\pi\otimes Q\), write
\[
q=P_0(E)=\rho_{\Delta,Q}(\widehat t).
\]
By data processing,
\[
\mathcal I_\pi(Q)=D_{\mathrm{KL}}(P\|P_0)
\ge
D_{\mathrm{KL}}(\operatorname{Bern}(p)\|\operatorname{Bern}(q)).
\]
The elementary inequality
\[
D_{\mathrm{KL}}(\operatorname{Bern}(p)\|\operatorname{Bern}(q))
\ge p\log\frac1q-\log2
\]
gives
\[
p\le\frac{\mathcal I_\pi(Q)+\log2}{\log(1/q)}.
\]
Therefore
\[
P\{d(t,\widehat t(Y))\ge\Delta\}=1-p
\ge
1-\frac{\mathcal I_\pi(Q)+\log2}{\log(1/\rho_{\Delta,Q}(\widehat t))},
\]
with the positive part inserted to make the bound nonnegative. Multiplying by \(\Delta\) and optimizing over \(\Delta,Q\) proves \eqref{eq:bayesian-algorithmic-lower-risk}.
\end{proof}

\begin{proof}[Proof of Theorem~\ref{thm:nn-comparison}]
For the Gaussian location experiment and the reference law \(Q_{\tau,m}\),
\[
\mathcal I_\pi(Q_{\tau,m})
=\int_T D_{\mathrm{KL}}\!\left(N(v_t,\tau^2I)\|N(m,\tau^2I)\right)\pi(dt)
=\frac{1}{2\tau^2}\int_T\|v_t-m\|^2\pi(dt)
=\mathcal I_{\pi,m}.
\]
Applying Proposition~\ref{prop:bayesian-algorithmic-lower} to \(\widehat\Theta\) and \(Q_{\tau,m}\) gives \eqref{eq:bayesian-nn-risk-main} under condition \eqref{eq:bayesian-nn-condition}.

It remains to relate this Bayes location error to a Gaussian range.  By the approximate comparison-decoder condition \eqref{eq:comparison-decoder-condition},
\[
\|Y-v_{\widehat\Theta}\|^2-\|Y-v_\Theta\|^2
\le
\Omega(\Theta)-\Omega(\widehat\Theta)+\varepsilon_{\rm opt}
\le \varepsilon_{\rm cmp}.
\]
Since \(Y=v_\Theta+\tau Z\), expanding and cancelling \(\tau^2\|Z\|^2\) gives
\[
\|v_\Theta-v_{\widehat\Theta}\|^2
+2\tau\langle Z,v_\Theta-v_{\widehat\Theta}\rangle
\le
\varepsilon_{\rm cmp},
\]
and hence
\[
\|v_\Theta-v_{\widehat\Theta}\|^2
\le
2\tau\langle Z,v_{\widehat\Theta}-v_\Theta\rangle
+\varepsilon_{\rm cmp}
=2\tau\bigl(X_{\widehat\Theta}-X_\Theta\bigr)+\varepsilon_{\rm cmp}.
\]
Taking expectations and using Jensen's inequality,
\[
c^2\Delta^2
\le
\left(\mathbb E d(\Theta,\widehat\Theta)\right)^2
\le
\mathbb E d(\Theta,\widehat\Theta)^2
\le
2\tau\,\mathbb E\bigl[X_{\widehat\Theta}-X_\Theta\bigr]+\mathbb E\varepsilon_{\rm cmp}.
\]
Rearranging proves \eqref{eq:nn-to-gaussian-range-main}.
\end{proof}

\begin{proof}[Proof of Corollary~\ref{coro:estimation}]
By the approximate residual comparison \eqref{eq:decoder-residual-comparison-for-upper}
and the identity \(Y=v_\Theta+\tau Z\),
\[
\|v_\Theta+\tau Z-v_{\widehat\Theta}\|^2
\le
\|\tau Z\|^2+\varepsilon_{\rm up}.
\]
Expanding the left-hand side and cancelling the common term
\(\tau^2\|Z\|^2\), we obtain the pathwise inequality
\[
d(\Theta,\widehat\Theta)^2
=
\|v_\Theta-v_{\widehat\Theta}\|^2
\le
2\tau\langle Z,v_{\widehat\Theta}-v_\Theta\rangle
+\varepsilon_{\rm up}
=
2\tau\bigl(X_{\widehat\Theta}-X_\Theta\bigr)+\varepsilon_{\rm up}.
\]
On the event in \eqref{eq:generic-gaussian-envelope-for-upper}, which holds
with probability at least \(1-\delta\) and is uniform over all \(t\in T\),
\[
X_{\widehat\Theta}-X_\Theta
\le
|X_{\widehat\Theta}|+|X_\Theta|
\le
A(\widehat\Theta)+A(\Theta).
\]
Hence, on the same event,
\[
d(\Theta,\widehat\Theta)^2
\le
2\tau\bigl(A(\Theta)+A(\widehat\Theta)\bigr)+\varepsilon_{\rm up},
\]
which proves \eqref{eq:algorithmic-upper-high-prob}.

For the expectation bound, let \(E\) denote the event
\(\{\forall t\in T:\ |X_t|\le A(t)\}\). On \(E\), the preceding display gives
\[
d(\Theta,\widehat\Theta)
\le
\sqrt{2\tau\bigl(A(\Theta)+A(\widehat\Theta)\bigr)+\varepsilon_{\rm up}}.
\]
On \(E^c\), we use the trivial bound
\[
d(\Theta,\widehat\Theta)\le D_T.
\]
Therefore
\[
\begin{aligned}
\mathbb E d(\Theta,\widehat\Theta)
&\le
\mathbb E\!\left[
\sqrt{2\tau\bigl(A(\Theta)+A(\widehat\Theta)\bigr)+\varepsilon_{\rm up}}\,\mathbf 1_E
\right]
+
D_T\,\mathbb P(E^c)\\
&\le
\mathbb E\sqrt{2\tau\bigl(A(\Theta)+A(\widehat\Theta)\bigr)+\varepsilon_{\rm up}}
+\delta D_T ,
\end{aligned}
\]
which proves \eqref{eq:algorithmic-upper-risk}. The final claim follows by
applying the absolute pointwise Gaussian envelope
\eqref{eq:absolute-pointwise-gaussian} to the anchored process
\((X_t)_{t\in T\cup\{t_0\}}\) and then restricting back to \(T\).
\end{proof}

\section{Proof of Separation Example and Pointwise Validation}
\label{sec:hub-cloud-algorithmic-ghost}

This section proves Theorem~\ref{thm:decision-aligned-separation}.  The example is a finite-dimensional weighted-basis construction.  It is intentionally elementary, but it captures the benchmark separation needed for the paper: the full ambient class is overparameterized, the correct subatlas is decision-dependent, and the actual estimator is unrestricted over the full class but uses a minimum-norm inductive bias.  The proof  keeps the discussion in separate paragraphs because each paragraph isolates one benchmark in the separation.

\begin{proof}[Proof of Theorem~\ref{thm:decision-aligned-separation}]
Let \(M,N\ge2\) satisfy \(\log N\ge1024\log M\).  Let
\[
 e_0,e_1,\ldots,e_M,f_1,\ldots,f_N
\]
be orthonormal vectors in \(\mathbb R^{M+N+1}\), and define
\[
\tau:=\frac{R}{\sqrt{\log M}},
\qquad
S:=\frac18\tau\sqrt{\log N}
=
\frac{R}{8}\sqrt{\frac{\log N}{\log M}} .
\]
The assumption on \(M,N\) implies \(S\ge4R\).  Set
\[
H_M:=\{Re_0,Re_1,\ldots,Re_M\},
\qquad
C_N:=\{Sf_1,\ldots,Sf_N\},
\qquad
T_{M,N}:=H_M\cup C_N,
\]
with Euclidean metric \(d(s,t)=\|s-t\|_2\).  The prior is
\[
\pi(Re_0)=\frac12,
\qquad
\pi(Re_i)=\frac1{2M},\quad i=1,\ldots,M,
\qquad
\pi(Sf_j)=0,
\quad j=1,\ldots,N,
\]
and we take \(\Delta=R/3\).  Since \(S\ge4R\), every \(\Delta\)-ball in \(T_{M,N}\) is a singleton.  Hence
\[
q_\pi(\Delta)=\sup_{a\in T_{M,N}}\pi(B_d(a,\Delta))=\frac12 .
\]
This proves the fixed-prior small-ball claim in part \(\mathrm{(i)}\).

\paragraph{Algorithmic ghost mass versus worst-case Fano mass.}
Let
\[
Y=\Theta+\tau Z,
\qquad
\Theta\sim\pi,
\qquad
Z\sim N(0,I_{M+N+1}),
\]
and use the ghost/reference law \(Q_\tau=N(0,\tau^2I_{M+N+1})\).  Under \(Q_\tau\), the scores of the \(M+1\) lower-norm points in \(H_M\) are exchangeable.  We first show that the cloud is selected with negligible ghost probability.  If a cloud point \(Sf_j\) beats the hub point \(Re_0\), then
\[
2\langle Y,Sf_j\rangle-(1+\lambda)S^2
\ge
2\langle Y,Re_0\rangle-(1+\lambda)R^2 .
\]
Under \(Q_\tau\), the random part of the difference is centered Gaussian with variance at most \(4\tau^2(S^2+R^2)\).  Since \(S\ge4R\) and \(\lambda=64\), a Gaussian union bound gives
\begin{align}\label{eq:selected}
Q_\tau(\widehat\Theta_\lambda\in C_N)\le N^{-2}
\end{align}
for all sufficiently large \(N\).  Conditional on \(\widehat\Theta_\lambda\in H_M\), exchangeability gives
\[
Q_\tau(\widehat\Theta_\lambda=Re_0\mid \widehat\Theta_\lambda\in H_M)=\frac1{M+1}.
\]
Consequently,
\[
Q_\tau(\widehat\Theta_\lambda=Re_0)
\le
\frac1{M+1}+N^{-2}.
\]
Since every \(\Delta\)-ball is a singleton and \(\pi\) vanishes on the cloud,
\[
\begin{aligned}
\rho_{\Delta,Q_\tau}(\widehat\Theta_\lambda)
&=
\mathbb E_{Y\sim Q_\tau}\pi(B_d(\widehat\Theta_\lambda(Y),\Delta))\\
&\le
\frac12 Q_\tau(\widehat\Theta_\lambda=Re_0)
+
\frac1{2M}Q_\tau(\widehat\Theta_\lambda\in\{Re_1,\ldots,Re_M\})
\le
\frac{2}{M}
\end{aligned}
\]
for all large \(M,N\).  Hence
\[
\log\frac1{\rho_{\Delta,Q_\tau}(\widehat\Theta_\lambda)}
\ge
\log M-O(1),
\]
whereas the worst-case relaxation for the same prior gives only \(\log2\).  This proves part \(\mathrm{(ii)}\) and the entropy separation in part \(\mathrm{(i)}\).

\paragraph{Sharp algorithmic lower bound for the actual unrestricted estimator.}
The information radius under \(Q_\tau\) is
\[
\mathsf I_{\pi,Q_\tau}
=
\mathbb E_{\Theta\sim\pi}
D_{\rm KL}\bigl(N(\Theta,\tau^2I)\,\|\,N(0,\tau^2I)\bigr)
=
\frac{\mathbb E_\pi\|\Theta\|^2}{2\tau^2}
=
\frac{R^2}{2\tau^2}
=
\frac12\log M .
\]
Combining this with the ghost-mass bound, for all large \(M\),
\[
\mathsf I_{\pi,Q_\tau}+\log2
\le
\frac34
\log\frac1{\rho_{\Delta,Q_\tau}(\widehat\Theta_\lambda)}.
\]
Proposition~\ref{prop:bayesian-algorithmic-lower} therefore yields
\begin{align}\label{eq:lower-Bayes}
\mathbb E_{\Theta,Y}d(\Theta,\widehat\Theta_\lambda(Y))
\ge
c\Delta
\asymp R .
\end{align}
This lower bound is sharp for the actual Bayesian decision problem.  When \(\Theta\in H_M\), all true parameters have norm \(R\).  The same score comparison as above, now under \(Y=\Theta+\tau Z\), gives
\[
\mathbb P_{\Theta,Y}(\widehat\Theta_\lambda\in C_N)\le N^{-2}.
\]
On the complementary event, both \(\Theta\) and \(\widehat\Theta_\lambda\) lie in \(H_M\), whose diameter is at most \(\sqrt2R\).  Hence
\begin{align}\label{eq:upper-Bayes}
\mathbb E_{\Theta,Y}d(\Theta,\widehat\Theta_\lambda(Y))
\le
\sqrt2R+N^{-2}S
\lesssim R .
\end{align}
Combining the algorithmic lower bound \eqref{eq:lower-Bayes} and the algorithmic upper bound \eqref{eq:upper-Bayes} proves the Bayes risk statement  for the estimator $\widehat\Theta_\lambda$:
\begin{align*}
E_{\Theta,Y}d(\Theta,\widehat\Theta_\lambda(Y))\asymp R.
\end{align*}
This proves the statement about the Bayes risk of the actual estimator  in part \(\mathrm{(iii)}\).
\paragraph{Bayes-optimal risk is of the same order.}
For the displayed prior, the Bayes-optimal risk is also \(\asymp R\). The upper bound follows from \eqref{eq:upper-Bayes}.  For the lower
bound, condition on the event \(\mathcal L\) that \(\Theta\) is one of the non-hub low-norm points.  Under this
conditional experiment, write
\[
  J\sim \mathrm{Unif}\{1,\ldots,M\},\qquad \Theta=Re_J,\qquad
  Y=Re_J+\tau Z,\quad Z\sim N(0,I_{M+N+1}).
\]
Given any estimator \(\widehat\theta(Y)\), define the induced low-norm index classifier
\[
  \widehat J(Y)\in \arg\min_{1\le j\le M}\|\widehat\theta(Y)-Re_j\|,
\]
with ties broken arbitrarily.  Since the non-hub low-norm points are separated by \(\sqrt2R\), the event
\(\|\widehat\theta(Y)-Re_J\|<R/2\) implies \(\widehat J=J\).  Hence
\[
  \mathbb E\!\left[\|\widehat\theta(Y)-\Theta\|\mid \mathcal L\right]
  \ge {\frac R 2}\,
      \mathbb P(\widehat J\neq J\mid \mathcal L).
\]
Let \(P_j=N(Re_j,\tau^2 I_{M+N+1})\), let \(\bar P=M^{-1}\sum_{j=1}^M P_j\), and take the reference
law \(Q=N(0,\tau^2 I_{M+N+1})\).  Then
\[
  I(J;Y)
  ={\frac 1  M}\sum_{j=1}^M D_{\mathrm{KL}}(P_j\|\bar P)
  \le {\frac 1 M}\sum_{j=1}^M D_{\mathrm{KL}}(P_j\|Q)
  ={\frac{R^2}{2\tau^2}}
  ={\frac 1 2}\log M,
\]
where the last equality uses \(\tau=R/\sqrt{\log M}\).  Fano's inequality therefore gives
\[
  \mathbb P(\widehat J\neq J\mid \mathcal L)
  \ge 1-{\frac{I(J;Y)+\log 2}{\log M}}
  \ge c_0
\]
for a universal constant \(c_0>0\), for all sufficiently large \(M\).  Since
\(\mathbb P(\mathcal L)=1/2\) under the original prior,
\[
  \mathbb E\|\widehat\theta(Y)-\Theta\|
  \ge {\frac 1 2}\,
      \mathbb E\!\left[\|\widehat\theta(Y)-\Theta\|\mid \mathcal L\right]
  \ge cR .
\]
Taking the infimum over \(\widehat\theta\) gives the Bayes-risk lower bound \(cR\). 
Thus the
class-wide Fano relaxation below is vacuous only as a validation; it does not reflect an easy
Bayes problem. 

By allowing the hub mass to be placed at different points of \(H_M=\{Re_0,Re_1,\ldots,Re_M\}\), rather than only at the original hub \(Re_0\), the same rotationally symmetric class construction can, for sufficiently large \(M\), accommodate arbitrarily large finite families of priors that exhibit the separation and have markedly different local mass profiles. Consequently, the Bayes-optimal oracle may vary substantially from prior to prior, while the unrestricted estimator remains unchanged. This completes the formal Bayes-optimal-risk statement in part \(\mathrm{(iii)}\).

\paragraph{Usual Fano relaxation is vacuous for the fixed prior.}
Replacing the exact ghost mass by the worst-case relaxation \(q_\pi(\Delta)=1/2\) gives
\[
\Delta
\left[
1-
\frac{\mathsf I_{\pi,Q_\tau}+\log2}{\log(1/q_\pi(\Delta))}
\right]_+
=
\Delta
\left[
1-
\frac{\mathsf I_{\pi,Q_\tau}+\log2}{\log2}
\right]_+
=0 .
\]
Thus, for the same prior, channel, and estimator, the usual fixed-prior Fano relaxation is vacuous, while the exact algorithmic ghost mass gives a sharp Bayes estimation lower bound.

\paragraph{Full-class minimax risk has arbitrarily large gap.}
Let
\[
\mathfrak R_{\rm mm}(T_{M,N},\tau)
:=
\inf_{\widehat\theta}
\sup_{\theta\in T_{M,N}}
\mathbb E_\theta d(\theta,\widehat\theta(Y)),
\]
where the infimum is over all measurable estimators with values in the ambient Euclidean space.  The trivial estimator \(\widehat\theta\equiv0\) gives \(\mathfrak R_{\rm mm}(T_{M,N},\tau)\le S\).  Conversely, restrict the parameter to the cloud \(C_N\) and put the uniform prior on \(C_N\).  For any estimator \(\widehat\theta\), define
\[
\widehat J(\widehat\theta)
\in
\arg\min_{1\le j\le N}\|\widehat\theta-Sf_j\|_2 .
\]
If \(d(Sf_J,\widehat\theta)<S/2\), then \(\widehat J(\widehat\theta)=J\), since the balls \(B(Sf_j,S/2)\) are disjoint.  Therefore
\[
\mathbb E_J d(Sf_J,\widehat\theta(Y))
\ge
\frac S2\,\mathbb P(\widehat J\ne J).
\]
By Fano's inequality with reference law \(N(0,\tau^2I_{M+N+1})\),
\[
I(J;Y)
\le
\frac{S^2}{2\tau^2}
=
\frac{1}{128}\log N .
\]
Thus, for all sufficiently large \(N\),
\[
\mathbb P(\widehat J\ne J)
\ge
1-
\frac{I(J;Y)+\log2}{\log N}
\ge
\frac12 .
\]
It follows that
\[
\mathfrak R_{\rm mm}(T_{M,N},\tau)
\ge
cS .
\]
Together with the trivial upper bound, this gives \(\mathfrak R_{\rm mm}(T_{M,N},\tau)\asymp S\), proving part \(\mathrm{(iv)}\).  Since \(S/R=(1/8)\sqrt{\log N/\log M}\), the full-class minimax risk is much larger than the Bayes risk whenever \(\log N/\log M\to\infty\).  The minimax scale is governed by an ambient cloud that is irrelevant to the specified Bayesian decision problem.

\paragraph{Algorithmically induced Gaussian validation matches the selected subatlas.}
Now define
\[
X_t:=\langle Z,t\rangle,
\qquad t\in T_{M,N},
\]
using the same standard Gaussian vector \(Z\) as in the location experiment.  We verify the approximate comparison condition in Theorem~\ref{thm:nn-comparison} with zero error.  The decoder minimizes \(\|Y-a\|^2+\lambda\|a\|^2\), so the first inequality in \eqref{eq:comparison-decoder-condition} holds with \(\Omega(t)=\lambda\|t\|^2\) and \(\varepsilon_{\rm opt}=0\).  Since \(\Theta\in H_M\) almost surely, \(\|\Theta\|=R\), whereas every output of \(\widehat\Theta_\lambda\) has norm either \(R\) or \(S\ge R\).  Hence \(\Omega(\widehat\Theta_\lambda)\ge\Omega(\Theta)\) almost surely, so \(\varepsilon_{\rm pen}=0\) and \(\varepsilon_{\rm cmp}=0\).  Applying Theorem~\ref{thm:nn-comparison} gives
\[
\mathcal G_{\rm alg}
:=
\mathbb E_{\Theta,Y}\bigl[X_{\widehat\Theta_\lambda(Y)}-X_\Theta\bigr]
\gtrsim
R\sqrt{\log M} .
\]
Conversely, on the event \(\widehat\Theta_\lambda\in H_M\), both \(\Theta\) and \(\widehat\Theta_\lambda\) lie in \(H_M\), and
\[
\mathbb E\sup_{s,u\in H_M}(X_s-X_u)\asymp R\sqrt{\log M}.
\]
The event \(\widehat\Theta_\lambda\in C_N\) has probability at most \(N^{-2}\) under the true prior.  By Cauchy--Schwarz and the standard Gaussian maximal bound on \(T_{M,N}\), its contribution is negligible.  Therefore
\[
\mathcal G_{\rm alg}
\asymp
\mathbb E\sup_{s,u\in H_M}(X_s-X_u)
\asymp
R\sqrt{\log M},
\]
which proves the first assertion in part \(\mathrm{(v)}\).

This also matches the pointwise Gaussian upper envelope on the selected subatlas.  By \eqref{eq:selected}, we already have
\[
  Q_\tau\{\widehat\Theta_\lambda(Y)\in C_N\}\le N^{-2}.
\]
for all sufficiently large $N$.
Let
\[
\nu_Q^H(a):=
Q_\tau\{\widehat\Theta_\lambda(Y)=a\mid \widehat\Theta_\lambda(Y)\in H_M\},
\qquad a\in H_M,
\]
be the ghost selection law conditioned on the lower-norm subatlas.  By exchangeability, \(\nu_Q^H\) is uniform on \(H_M\).  Since the restricted process on \(H_M\) consists of \(M+1\) Gaussian variables of variance at most \(R^2\), Theorem~\ref{thm:pointwise-gaussian} gives, up to universal constants and confidence logarithms,
\[
\mathbb P\!\left(\forall t\in H_M:\ |X_t|\le C R\sqrt{\log\frac{M}{\delta}}\right)
\ge 1-\delta .
\]
This proves the pointwise-envelope assertion in part \(\mathrm{(v)}\).

\paragraph{Global Gaussian supremum has arbitrarily large gap.}
The full ambient Gaussian field has a larger global scale because of the cloud.  Comparing the cloud points with the fixed hub \(Re_0\),
\[
\mathbb E\sup_{s,u\in T_{M,N}}(X_s-X_u)
\ge
\mathbb E\!\left[\max_{1\le j\le N}X_{Sf_j}-X_{Re_0}\right]
=
S\,\mathbb E\max_{1\le j\le N}Z_j
\gtrsim
S\sqrt{\log N}.
\]
Using the definition of \(S\),
\[
S\sqrt{\log N}
=
\frac18 R\frac{\log N}{\sqrt{\log M}} .
\]
This proves part \(\mathrm{(vi)}\).  Comparing with the decision-aligned scale \(R\sqrt{\log M}\), the gap is of order \(\log N/\log M\), and taking \(N=M^K\) makes the ratio of order \(K\).

\end{proof}

\paragraph{Summary and interpretation.}
The construction gives a fully explicit sequence of finite-dimensional Gaussian location problems for which
\begin{align*}
&\text{fixed-prior Fano relaxation}\\
\quad <\!\!< \quad
&\text{algorithmic ghost-mass lower envelope}\\
\quad \asymp \quad
&\text{pointwise upper envelope on the selected subatlas}\\
\quad <\!\!< \quad
&\text{full-class minimax/global Gaussian scale}.
\end{align*}
The example does not claim that minimax theory is wrong or fundamentally insufficient.  Rather, it shows that a full-class minimax benchmark can be the wrong validation for a pointwise decision problem, while a restricted minimax benchmark can be oracle-dependent.  The norm-regularized decoder uses a fixed numerical penalty \(\lambda\), independent of the prior $\pi$ and the problem parameters \(M,N,R,S,\tau\), while identifying the low-norm hub subatlas \(H_M\) as the relevant subatlas is precisely the decision-dependent information revealed by the ghost image \(Y\mapsto\widehat\Theta_\lambda(Y)\).  The algorithmic ghost-mass lower bound supplies the missing validation: it is tied to the fixed prior, the actual estimator, and the subgeometry selected by that estimator.  This is the finite-dimensional analogue of the pointwise-complexity viewpoint in deep networks, where the effective subatlas is induced by the learned representation and compressed spectrum rather than specified as a fixed restricted class in advance.

Here we keep the cleanest finite formulation.  Finite sets are already the standard setting for Sudakov, Fano, and packing lower bounds, and they make the quantifiers transparent: the full ambient class, the oracle restricted subatlas, the fixed prior, and the actual estimator are all explicit.  Extensions to continuous settings, as well as
computational-efficiency questions concerning the full class versus selected subatlases, are left to
future work.

\section{Proof of Penalty-Range Information-Relaxation Reformulation}\label{sec:minimax}

\begin{proof}[Proof of Theorem~\ref{thm:penalty-range-information-relaxation}]
Write the regularized score as
\[
G_\lambda(y,a)
:=
2\langle y,a\rangle-(1+\lambda)\|a\|_2^2 .
\]
The decoder \eqref{eq:lambda-family-penalty-revelation} is equivalently a maximizer of \(G_\lambda(Y,a)\) over \(a\in T_{M,N}\).  Throughout the proof we use
\[
\frac{S}{\tau}=\frac{\sqrt{\log N}}{8},
\qquad
\frac{R}{S}=8\sqrt{\frac{\log M}{\log N}}=o(1).
\]
Let \(r_{M,N}:=R/S\).  Then \(r_{M,N}\to0\).  The critical constant
\(\lambda_{\rm c}=16\sqrt2-1\) comes from the cloud maximum, because
\[
\frac{2\tau S\sqrt{2\log N}}{S^2}=16\sqrt2.
\]

\paragraph{Optimal penalty range suppresses the cloud.}
Fix \(\eta\in(0,\lambda_{\rm c})\), and take \(\lambda\ge\lambda_{\rm c}+\eta\).  First work under the ghost law \(Q_\tau=N(0,\tau^2I)\).  If a cloud point is selected, then for some \(j\le N\), the point \(Sf_j\) beats the hub point \(Re_0\).  Hence
\[
2\tau(SZ_j-RZ_0)
\ge
(1+\lambda)(S^2-R^2),
\]
where \(Z_j,Z_0\) are independent standard Gaussians.  The left-hand side is centered Gaussian with variance
\[
4\tau^2(S^2+R^2).
\]
Therefore, by the Gaussian tail bound and a union bound,
\begin{align*}
Q_\tau\{\widehat\Theta_\lambda(Y)\in C_N\}
&\le
N\exp\!\left(
-\frac{(1+\lambda)^2(S^2-R^2)^2}{8\tau^2(S^2+R^2)}
\right).
\end{align*}
Since \(1+\lambda\ge16\sqrt2+\eta\), \(S/\tau=\sqrt{\log N}/8\), and \(R/S=o(1)\), there is a number \(b_\eta>0\) such that, for all sufficiently large \(M,N\),
\[
\frac{(1+\lambda)(S^2-R^2)}{2\tau\sqrt{S^2+R^2}}
\ge
(\sqrt2+b_\eta)\sqrt{\log N}.
\]
The preceding display gives
\[
Q_\tau\{\widehat\Theta_\lambda(Y)\in C_N\}
\le
N\exp\{-(\sqrt2+b_\eta)^2\log N/2\}
\le
N^{-a_\eta}
\]
for some \(a_\eta>0\).

The same argument is uniform under every true parameter \(\theta\in H_M\).  Write \(\theta=Re_i\).  If \(Sf_j\) is selected, then it beats the true point \(Re_i\), so
\[
2\tau(SZ_j-RZ_i)
\ge
(1+\lambda)S^2+(1-
\lambda)R^2.
\]
The right-hand side equals \((1+\lambda)S^2-(\lambda-1)R^2\), and after division by \(2\tau\sqrt{S^2+R^2}\) it is again at least \((\sqrt2+b_\eta)\sqrt{\log N}\) for all sufficiently large \(M,N\), uniformly over \(\lambda\ge\lambda_{\rm c}+\eta\).  A union bound over the \(N\) cloud points gives
\[
\sup_{\theta\in H_M}P_\theta\{\widehat\Theta_\lambda(Y)\in C_N\}
\le N^{-a_\eta}.
\]
Taking the supremum over \(\lambda\in\Lambda_{\rm opt}(\eta)\) proves \eqref{eq:opt-penalty-cloud-suppression}.

We next derive the ghost-mass bound.  On the event \(\widehat\Theta_\lambda(Y)\in H_M\), the \(M+1\) hub scores are exchangeable under \(Q_\tau\), because each has the form
\[
2\tau R Z_k-(1+\lambda)R^2,
\qquad k=0,1,\ldots,M.
\]
The tie probability is zero.  Hence, conditionally on selecting a point of \(H_M\), the selected hub point is uniform on \(H_M\).  Since every \(R/3\)-ball is a singleton and \(\pi\) gives mass \(1/2\) to \(Re_0\), mass \(1/(2M)\) to each \(Re_i\), \(1\le i\le M\), and zero mass to \(C_N\),
\begin{align*}
\rho_{\Delta,Q_\tau}(\widehat\Theta_\lambda)
&=
\mathbb E_{Y\sim Q_\tau}\pi\bigl(B_d(\widehat\Theta_\lambda(Y),\Delta)\bigr)\\
&\le
\frac12\left(\frac{1}{M+1}+N^{-a_\eta}\right)
+
\frac{1}{2M}
+N^{-a_\eta}.
\end{align*}
Because \(\log N/\log M\to\infty\), we have \(N^{-a_\eta}=o(M^{-1})\).  Thus \(\rho_{\Delta,Q_\tau}(\widehat\Theta_\lambda)\le C/M\), proving \eqref{eq:opt-penalty-ghost-mass}.

\paragraph{Under-regularization selects the cloud.}
Now fix \(\lambda\le\lambda_{\rm c}-\eta\) and \(\theta\in H_M\).  Write \(Y=\theta+\tau Z\).  Let
\[
\mathcal E_C:=\left\{\max_{1\le j\le N} Z_j^C\ge(\sqrt2-\delta)\sqrt{\log N}\right\},
\qquad
\mathcal E_H:=\left\{\max_{0\le k\le M} Z_k^H\le2\sqrt{\log M}\right\},
\]
where \(Z_j^C=\langle Z,f_j\rangle\), \(Z_k^H=\langle Z,e_k\rangle\), and \(\delta>0\) will be chosen below as a function of \(\eta\).  The standard Gaussian maximum bounds give
\[
P(\mathcal E_C)\to1,
\qquad
P(\mathcal E_H)\to1.
\]
Indeed, the first follows from
\[
P(\mathcal E_C^c)
=\Phi((\sqrt2-\delta)\sqrt{\log N})^N
\le
\exp\{-c_\delta N^{\sqrt2\delta-
\delta^2/2}/\sqrt{\log N}\}\to0,
\]
while the second follows from
\[
P(\mathcal E_H^c)
\le
(M+1)\exp(-2\log M)\to0.
\]
On \(\mathcal E_C\), the best cloud score is at least
\begin{align*}
\max_{1\le j\le N}G_\lambda(Y,Sf_j)
&=
2\tau S\max_{1\le j\le N}Z_j^C-(1+\lambda)S^2 \\
&\ge
\bigl(16(\sqrt2-\delta)-(1+\lambda)\bigr)S^2
\ge
(\eta-16\delta)S^2 .
\end{align*}
Choose \(\delta=\eta/32\).  Then the last lower bound is \((\eta/2)S^2\).  On \(\mathcal E_H\), every hub point \(Re_k\) has score at most
\begin{align*}
G_\lambda(Y,Re_k)
&=2\langle\theta,Re_k\rangle+2\tau R Z_k^H-(1+\lambda)R^2\\
&\le
2R^2+4\tau R\sqrt{\log M}\\
&=
6R^2
=
6r_{M,N}^2S^2.
\end{align*}
Since \(r_{M,N}\to0\), this is smaller than \((\eta/2)S^2\) for all sufficiently large \(M,N\).  Therefore, on \(\mathcal E_C\cap\mathcal E_H\), a cloud point strictly beats every hub point.  Thus
\[
\inf_{\lambda\in\Lambda_{\rm under}(\eta)}
\inf_{\theta\in H_M}
P_\theta\{\widehat\Theta_\lambda(Y)\in C_N\}
\ge
P(\mathcal E_C\cap\mathcal E_H)\to1,
\]
which proves \eqref{eq:under-penalty-cloud-selection}.

\paragraph{Penalty-revealed algorithmic minimax risks.}
We first prove the lower bound \(\mathfrak R_{\rm RNN}^{H}(\Lambda_{\rm opt}(\eta))\gtrsim R\).  This lower bound holds even if one allows all measurable estimators, so it also holds for the smaller class of regularized nearest-neighbor rules.  Put the uniform prior on \(\{Re_1,\ldots,Re_M\}\), write \(\Theta=Re_J\), and let \(J\sim{\rm Unif}\{1,
\ldots,M\}\).  For an arbitrary estimator \(\widehat\theta(Y)\), define
\[
\widehat J(Y)\in\arg\min_{1\le j\le M}\|\widehat\theta(Y)-Re_j\|_2.
\]
If \(\|\widehat\theta(Y)-Re_J\|_2<R/2\), then \(\widehat J=J\).  With \(Q=N(0,\tau^2I)\),
\[
I(J;Y)
\le
\frac1M\sum_{j=1}^M
D_{\rm KL}\bigl(N(Re_j,\tau^2I)\|Q\bigr)
=
\frac{R^2}{2\tau^2}
=
\frac12\log M.
\]
Fano's inequality gives \(P(\widehat J\ne J)\ge c_0\) for a universal \(c_0>0\), for all sufficiently large \(M\).  Hence
\[
\sup_{\theta\in H_M}\mathbb E_\theta d(\theta,\widehat\theta(Y))
\ge
\mathbb E d(\Theta,\widehat\theta(Y))
\ge
\frac R2 P(\widehat J\ne J)
\ge cR.
\]
Taking the infimum over the smaller class \(\{\widehat\Theta_\lambda:\lambda\in\Lambda_{\rm opt}(\eta)\}\) gives the desired lower bound.

For the matching upper bound, take any \(\lambda\in\Lambda_{\rm opt}(\eta)\).  By \eqref{eq:opt-penalty-cloud-suppression}, uniformly over \(\theta\in H_M\), the event \(\widehat\Theta_\lambda(Y)\in C_N\) has probability at most \(N^{-a_\eta}\).  On its complement, both \(\theta\) and \(\widehat\Theta_\lambda(Y)\) lie in \(H_M\), whose diameter is at most \(\sqrt2R\).  On the exceptional event, the loss is at most \(S+R\le2S\).  Therefore
\[
\sup_{\theta\in H_M}\mathbb E_\theta d(\theta,\widehat\Theta_\lambda(Y))
\le
\sqrt2R+2S N^{-a_\eta}
\le CR,
\]
because \(S/R=(1/8)\sqrt{\log N/\log M}=N^{o(1)}\), whereas \(N^{-a_\eta}\) is polynomially small in \(N\).  This proves \(\mathfrak R_{\rm RNN}^{H}(\Lambda_{\rm opt}(\eta))\asymp R\).

For \(\lambda\in\Lambda_{\rm under}(\eta)\), \eqref{eq:under-penalty-cloud-selection} implies, uniformly over \(\theta\in H_M\), that the decoder selects a cloud point with probability tending to one.  Whenever \(\theta\in H_M\) and \(a\in C_N\),
\[
d(\theta,a)=\sqrt{R^2+S^2}\ge S.
\]
Thus
\[
\inf_{\lambda\in\Lambda_{\rm under}(\eta)}
\sup_{\theta\in H_M}
\mathbb E_\theta d(\theta,\widehat\Theta_\lambda(Y))
\ge
(1-o(1))S.
\]
The reverse inequality follows from the deterministic bound \(d(\theta,a)\le S+R\le2S\) for \(\theta\in H_M\), \(a\in T_{M,N}\).  This proves the under-regularized part of \eqref{eq:penalty-minimax-gap}.

\paragraph{Full-class minimax risk.}
The full minimax calculation is unchanged from Theorem~\ref{thm:decision-aligned-separation}.  The constant estimator \(0\) gives the upper bound \(\mathfrak R_{\rm mm}^{\rm full}\le S\).  Conversely, put the uniform prior on \(C_N\), write \(\Theta=Sf_J\), and let \(J\sim{\rm Unif}\{1,
\ldots,N\}\).  If an estimator has loss less than \(S/2\), then nearest-neighbor decoding among \(\{Sf_1,\ldots,Sf_N\}\) recovers \(J\).  Moreover,
\[
I(J;Y)
\le
\frac1N\sum_{j=1}^N
D_{\rm KL}\bigl(N(Sf_j,\tau^2I)\|N(0,\tau^2I)\bigr)
=
\frac{S^2}{2\tau^2}
=
\frac1{128}\log N.
\]
Fano's inequality gives a classification-error probability bounded below by a positive universal constant.  Therefore the minimax risk over \(T_{M,N}\) is at least \(cS\), proving \eqref{eq:penalty-full-minimax}.

\paragraph{Gaussian ranges and comparison lower bound.}
On \(H_M\), the field consists of \(M+1\) independent centered Gaussians with variance \(R^2\).  Hence
\[
\Gamma_H
=
\mathbb E\sup_{s,u\in H_M}(X_s-X_u)
\asymp R\sqrt{\log M}.
\]
For the full class,
\[
\mathbb E\sup_{t\in T_{M,N}}|X_t|
\le
C\bigl(R\sqrt{\log M}+S\sqrt{\log N}\bigr)
\le
C'S\sqrt{\log N},
\]
and the matching lower bound follows from the cloud alone:
\[
\mathbb E\sup_{s,u\in T_{M,N}}(X_s-X_u)
\ge
\mathbb E\max_{1\le j\le N}X_{Sf_j}-\mathbb E X_{Re_0}
\gtrsim
S\sqrt{\log N}.
\]
Thus \(\Gamma_{\rm full}\asymp S\sqrt{\log N}\).

It remains to connect the penalty-revealed range to Theorem~\ref{thm:nn-comparison}.  Fix \(\lambda\in\Lambda_{\rm opt}(\eta)\).  The decoder exactly minimizes \(\|Y-a\|^2+\lambda\|a\|^2\), so the first inequality in \eqref{eq:comparison-decoder-condition} holds with \(\Omega(t)=\lambda\|t\|^2\) and \(\varepsilon_{\rm opt}=0\).  Under the prior \(\pi\), \(\Theta\in H_M\) almost surely, hence \(\|\Theta\|=R\).  Every possible output has norm either \(R\) or \(S\ge R\), so \(\Omega(\widehat\Theta_\lambda)\ge\Omega(\Theta)\), and the second inequality in \eqref{eq:comparison-decoder-condition} holds with \(\varepsilon_{\rm pen}=0\).  Therefore \(\varepsilon_{\rm cmp}=0\).

With the centered ghost law \(Q_\tau=N(0,\tau^2I)\), the information radius is
\[
\mathcal I_{\pi,0}
=
\frac{\mathbb E_\pi\|\Theta\|^2}{2\tau^2}
=
\frac{R^2}{2\tau^2}
=
\frac12\log M.
\]
By \eqref{eq:opt-penalty-ghost-mass},
\[
\log\frac1{\rho_{\Delta,Q_\tau}(\widehat\Theta_\lambda)}
\ge
\log M-O(1).
\]
Thus, for all sufficiently large \(M\), the information condition \eqref{eq:bayesian-nn-condition} holds, for example with \(c=1/4\).  Since \(\Delta=R/3\), Theorem~\ref{thm:nn-comparison} gives
\[
\mathcal G_{\rm alg}(\lambda)
\ge
c'\frac{\Delta^2}{\tau}
\asymp
R\sqrt{\log M}.
\]
For the reverse inequality, let \(A_\lambda=\{\widehat\Theta_\lambda(\Theta+\tau Z)\in C_N\}\).  On \(A_\lambda^c\), both \(\Theta\) and \(\widehat\Theta_\lambda\) lie in \(H_M\), so
\[
X_{\widehat\Theta_\lambda}-X_\Theta
\le
\sup_{s,u\in H_M}(X_s-X_u).
\]
On \(A_\lambda\), Cauchy--Schwarz and the standard second-moment bound for Gaussian maxima give
\[
\mathbb E\left[\sup_{s,u\in T_{M,N}}|X_s-X_u|\,\mathbf 1_{A_\lambda}\right]
\le
P(A_\lambda)^{1/2}
\left(\mathbb E\sup_{s,u\in T_{M,N}}|X_s-X_u|^2\right)^{1/2}
\le
C N^{-a_\eta/2}S\sqrt{\log N}
=o(R\sqrt{\log M}).
\]
Combining this with \(\Gamma_H\asymp R\sqrt{\log M}\) proves
\(\mathcal G_{\rm alg}(\lambda)\asymp R\sqrt{\log M}\), uniformly over \(\lambda\in\Lambda_{\rm opt}(\eta)\).  The displayed ratios follow by substituting \(S/R=(1/8)\sqrt{\log N/\log M}\), and the special case \(N=M^K\) gives gaps \(\sqrt K\) and \(K\).  This completes the proof.
\end{proof}

\section{Conclusion}\label{sec:summary}

The paper develops four linked pointwise and algorithmic messages, and shows how they yield separations from global complexity scales. First, for Gaussian processes, the natural refinement of generic chaining is a simultaneous pointwise envelope for the field, not only a scalar estimate on \(\mathbb E\sup_x X_x\). This envelope is sharp in expectation after optimizing over the prior and recovers the anchored Fernique--Talagrand scale. It should be useful whenever one wants to retain local field information before taking a final supremum.

Second, the Bayesian algorithmic lower envelope gives a single-radius lower-bound counterpart to pointwise complexity.  It is an information-theoretic Bayes-risk bound for every estimator \(\widehat t(Y)\), expressed through the exact ghost small-ball mass \(\mathbb E_Q\pi(B_d(\widehat t(Y),\Delta))\). The comparison-decoder theorem turns this Bayes distance obstruction into a lower bound on a decision-aligned Gaussian range; conversely, a simultaneous pointwise Gaussian envelope yields an algorithmic upper bound on the same decoding risk. Importantly, we  construct a weighted-basis ``hub--cloud'' example for norm-regularized nearest-neighbor decoder where the fixed-prior Fano relaxation is vacuous, the algorithmic ghost-mass lower envelope and pointwise upper envelope match on the estimator-selected subatlas, while the full-class minimax risk and global Gaussian scale are much larger. The key is to retain the algorithmic mass: it provides the local validation of pointwise complexity for fixed estimators in overparameterized ambient classes, precisely in regimes where classical minimax theory becomes either too coarse or oracle-dependent. 

Third, we reformulate the same separating example in minimax language as a penalty-range information-relaxation reformulation, which treats robust tuning information as a statistical resource for a fixed algorithmic family.  This points to a broader algorithmic information-gap viewpoint, connecting the offline separation here to the adaptivity gap of \citet{maiti-xu-jamieson-2026} and to classical high-dimensional regularization phenomena such as the Lasso.

Fourth, finite-cutoff renormalization and graph local time illustrate the same structural mechanism in settings motivated by physics and probability. The structural analogies between RG and DNN, and between pointwise dimension and graph local time, show that trajectory-aligned or selected-subatlas complexity can be intrinsically smaller than the scale obtained by passing through a full-cover event.

Taken together, the results suggest a unified principle: global worst-case objects capture genuine hardness of the full class, while pointwise and algorithmic objects can validate tractability on the local geometry selected by a field, an estimator, or a renormalization trajectory. Algorithm-suggested information relaxation provides a minimax language for reformulating the separation results and is relevant to algorithmic-robustness questions in traditional high-dimensional models. In this sense, the paper connects field-level generic chaining with decision- and trajectory-aligned validation, and gives explicit separations between pointwise and global complexity scales through reusable tools and transparent formulations.

\newpage
\appendix

\section{Finite-Cutoff Renormalization  Application}\label{sec:finite-rg}

\subsection{Structural analogy between finite-cutoff RG and feature-learning DNN}\label{subsec:main-rg-theorem}

This subsection makes precise the analogy between finite-cutoff renormalization and feature-learning DNN in \cite{li-xu-pointwise-generalization-dnn}, while also recording where the analogy stops. The analogy does not rest merely on the correspondence between composite layers and RG levels. More essentially, both settings exhibit a common \emph{feature/operator inner-product} structure within each composite layer of a DNN, or within each exponential-family RG level. This correspondence leads to the same proof mechanism: exact finite-step telescoping, a pointwise ellipsoidal metric induced by a learned or renormalized secant Gram, and a hierarchical atlas that separates local effective dimension from the global price of making the prior independent of the realized trajectory.

 The distinction is also important. A DNN layer has a rectangular matrix parameter and a common input-feature {\it matrix}, so its pointwise metric has a Kronecker repetition.  A Wilsonian RG step is the operation of integrating out short-distance or fluctuation variables and rewriting the resulting coarse marginal as an effective action.  After a finite operator truncation, this effective action is usually represented by a {\it vector} of couplings \(g_j=(g_{j,1},\ldots,g_{j,p_j})\), not by a learned matrix weight. We clarify this distinction after \eqref{eq:secant-Gram} and again after Theorem~\ref{thm:rg-pointwise-dimension}.

\begin{center}
\begin{tabular}{@{}ll@{}}
\toprule
Deep-network pointwise theory & Finite-cutoff RG analogue \\
\midrule
Layer index \(\ell\) & RG level index \(j\) \\
Weights matrix \(W_\ell\) & Level coupling vector \(g_j\) \\
Feature matrix \(F_{\ell-1}(W,X)\) & Renormalized operator family \(O_j(S_j)\) \\
Feature Gram \(F_{\ell-1}F_{\ell-1}^{\top}\) & Secant response Gram \(\Gamma_j^{\rm sec}\) \\
Exact layer telescoping & Exact replacement over RG levels \\
Outer Lipschitz \(M_{\ell\to L}\) & Stability of later RG maps \(M_{j\to J}\) \\
Active feature eigenspace & Relevant/marginal active operator subspace \\
Grassmannian prior & Phase/subspace atlas prior \\
Feature-rank compression & Contraction or truncation of irrelevant directions \\
Kernel/NTK baseline & Free or massive GFF baseline \\
\bottomrule
\end{tabular}
\end{center}

\paragraph{Terminology for radii, resolutions, and RG levels.}
The word ``scale'' appears in several neighboring literatures, so we keep the terminology separate.  The Bayesian lower envelope is a \emph{single-radius} statement at a testing radius \(\Delta\), whereas generic chaining is multiresolution and integrates over radii.  The Gaussian upper theorem also uses neighborhood radii, such as \(\varepsilon\), \(r_0\), and \(s_0\), only as metric resolutions.  In the renormalization appendix we use \emph{RG level} for the index of a block-spin or fluctuation integration, and \emph{finite-cutoff} or \emph{finite-step} for the non-infinitesimal nature of the telescoping argument.

\paragraph{Scope of the finite-cutoff appendix.}
The RG statements below are finite-dimensional, finite-cutoff complexity statements. They are compatible with rigorous Wilsonian RG frameworks \citep{brydges-slade-rg-step-2015,bauerschmidt-brydges-slade-2019}, but they do not attempt to prove continuum construction, scale-uniform Gaussianization, reflection positivity, or a Yang--Mills mass gap. Model-specific results, such as random-current estimates for four-dimensional Ising/\(\Phi^4_4\) \citep{aizenman-duminil-copin-2021} or gauge-field estimates for Yang--Mills \citep{jaffe-witten-yangmills,chatterjee-yangmills-probabilists}, remain separate inputs. The purpose here is only to show how pointwise local-chart/global-atlas complexity can be layered on top of a finite RG map when the needed stability, curvature, and truncation estimates are available.

\subsection{Pointwise complexity of composite renormalization maps}
\paragraph{Exact finite-step telescoping for composite renormalization maps.}

Let \(\mathcal S_0,\ldots,\mathcal S_J\) be finite-dimensional normed spaces.  For each RG level \(j\), let
\[
\mathcal R_j^{g_j}:\mathcal S_j\to\mathcal S_{j+1},
\qquad g_j\in B_2(R_j)\subset\mathbb R^{p_j},
\]
be a nonlinear coarse-graining map.  Starting from a fixed \(S_0\in\mathcal S_0\), define the RG trajectory
\[
S_{j+1}(g):=\mathcal R_j^{g_j}(S_j(g)),
\qquad
S_J(g)=\mathcal R_{J-1}^{g_{J-1}}\circ\cdots\circ\mathcal R_0^{g_0}(S_0).
\]
This notation covers, for example, the exact Wilsonian identity
\[
 e^{-S_{j+1}(\Phi)}=
 \int e^{-S_j(\Phi+\psi_j)}\,d\gamma_j(\psi_j).
\]
This is the finite-dimensional analogue of integrating out one fluctuation field at one RG level.  Perturbative RG expands this map; here we keep the finite map itself.

This is the sense in which the appendix uses Wilsonian RG \citep{brydges-slade-rg-step-2015,bauerschmidt-brydges-slade-2019}. Wilsonian RG refers to the coarse-graining transformation on actions or measures obtained by integrating out fluctuations. An exponential-family RG coordinate system gives a finite-dimensional projection or truncation of this transformation: after choosing operators $O_{j,a}$, one writes the retained effective action in natural parameters $g_{j,a}$. This  relation induces inner products analogous to those generated by feature maps in DNN. In the RG setting, however, the natural finite-dimensional coordinates are coupling vectors, not matrix-valued weights with shared-feature structure as in DNN.

For two coupling sequences \(g,g'\), define the hybrid trajectory
\[
S_j^{(g',g;m)}
:=
\mathcal R_{j-1}^{g'_{j-1}}\circ\cdots\circ\mathcal R_m^{g'_m}
\circ\mathcal R_{m-1}^{g_{m-1}}\circ\cdots\circ\mathcal R_0^{g_0}(S_0),
\]
with the evident convention at the endpoints.  Then the following identity is exact.

\begin{lemma}[Exact RG replacement telescoping]\label{lem:rg-telescoping}
For every \(g,g'\),
\begin{equation}\label{eq:rg-telescope}
S_J(g')-S_J(g)=
\sum_{j=0}^{J-1}
\Bigl[
\mathcal R_{J-1:j+1}^{g'}\bigl(\mathcal R_j^{g'_j}(S_j^{(j)})\bigr)
-
\mathcal R_{J-1:j+1}^{g'}\bigl(\mathcal R_j^{g_j}(S_j^{(j)})\bigr)
\Bigr],
\end{equation}
where \(S_j^{(j)}:=\mathcal R_{j-1}^{g_{j-1}}\circ\cdots\circ\mathcal R_0^{g_0}(S_0)\) and \(\mathcal R_{J-1:j+1}^{g'}:=\mathcal R_{J-1}^{g'_{J-1}}\circ\cdots\circ\mathcal R_{j+1}^{g'_{j+1}}\).
\end{lemma}

\begin{proof}
Insert the hybrid trajectories that use \(g\) up to level \(j-1\) and \(g'\) from level \(j\) onward.  The difference between two consecutive hybrids changes only the scale-\(j\) map.  Summing these consecutive differences gives \eqref{eq:rg-telescope}.
\end{proof}

Let \(\rho\) be a pseudometric (in particular, satisfying the triangle inequality) on \(\mathcal S_J\).  The next result is the RG analogue of the non-perturbative DNN feature expansion in \citet[Lemma 1]{li-xu-pointwise-generalization-dnn}: the finite replacement identity induces an ellipsoidal metric whose matrices are finite-scale secant response Grams.

\begin{theorem}[Secant-Gram metric domination for nonlinear RG]\label{thm:rg-secant-domination}
Fix \(g\) and a finite neighborhood \(\mathcal N(g,\varepsilon)\).  Assume that for each \(j\): 
\begin{enumerate}[label=(\roman*)]
\item the later RG flow is locally stable: for every \(g'\in\mathcal N(g,\varepsilon)\) and all states \(U,V\) arising along the corresponding hybrid trajectories,
\[
\rho\bigl(\mathcal R_{J-1:j+1}^{g'}(U),\mathcal R_{J-1:j+1}^{g'}(V)\bigr)
\le M_{j\to J}(g,\varepsilon)\|U-V\|_{j+1};
\]
\item the scale-\(j\) finite replacement is dominated by a positive semidefinite secant Gram \(\Gamma_j(g,\varepsilon)\):
\[
\bigl\|\mathcal R_j^{g'_j}(S_j^{(j)})-
\mathcal R_j^{g_j}(S_j^{(j)})\bigr\|_{j+1}^2
\le
(g'_j-g_j)^\top\Gamma_j(g,\varepsilon)(g'_j-g_j).
\]
\end{enumerate}
Then all \(g'\in\mathcal N(g,\varepsilon)\) satisfy
\begin{equation}\label{eq:rg-metric-domination}
\rho(S_J(g'),S_J(g))^2
\le
\sum_{j=0}^{J-1}
(g'_j-g_j)^\top G_j(g,\varepsilon)(g'_j-g_j),
\qquad
G_j:=J\,M_{j\to J}(g,\varepsilon)^2\Gamma_j(g,\varepsilon).
\end{equation}
\end{theorem}

\begin{proof}
Apply Lemma~\ref{lem:rg-telescoping}, the triangle inequality for \(\rho\), the local stability assumption, and the one-step secant-Gram bound. Cauchy--Schwarz gives a factor \(J\) when the \(J\) increments are summed; this factor is included in the definition of \(G_j\).
\end{proof}

\paragraph{How the secant Gram is verified in exponential-family RG steps.}
Condition (ii) in Theorem~\ref{thm:rg-secant-domination} is a finite-dimensional secant calculation.  Fix the state \(S_j^{(j)}\) and set
\[
F_j(g_j):=\mathcal R_j^{g_j}(S_j^{(j)}).
\]
If \(F_j\) is continuously Frechet differentiable on the finite neighborhood under consideration and the next-scale norm is Hilbertian, then for \(\Delta_j=g'_j-g_j\) and \(g_{j,t}=g_j+t\Delta_j\),
\[
F_j(g'_j)-F_j(g_j)=\int_0^1 DF_j(g_{j,t})[\Delta_j]dt,
\]
and Jensen's inequality gives
\begin{equation}\label{eq:secant-gram-verification}
\|F_j(g'_j)-F_j(g_j)\|_{j+1}^2
\le
\Delta_j^\top
\left(\int_0^1 DF_j(g_{j,t})^*DF_j(g_{j,t})dt\right)
\Delta_j.
\end{equation}
Thus condition (ii) holds with the pairwise secant Gram in parentheses, or with any positive-semidefinite Loewner envelope that dominates it uniformly over the finite neighborhood.  For an exponential-family RG step
\[
\mathcal R_j^g(S)(\Phi)=-\log\int\exp\Bigl\{-S(\Phi+\psi)-\sum_{a=1}^{p_j}g_aO_{j,a}(\Phi+\psi)\Bigr\}d\gamma_j(\psi),
\]
one has
\[
\partial_{g_a}\mathcal R_j^g(S)(\Phi)=\mathbb E_{j,\Phi,g}[O_{j,a}(\Phi+\psi)].
\]
Consequently the secant Gram is the Gram matrix of the finite-scale response functions
\begin{align}\label{eq:secant-Gram}
\Gamma^{\rm sec}_{j,ab}(g,g')=
\int_0^1\langle m_{j,a}^{t},m_{j,b}^{t}\rangle_{j+1}dt,
\qquad
m_{j,a}^{t}(\Phi):=\mathbb E_{j,\Phi,g_{j,t}}[O_{j,a}(\Phi+\psi)].
\end{align}
This is the precise RG analogue of the learned feature Gram matrix in the DNN pointwise metric. The word ``analogue'' here means ``response-Gram analogue'', not ``identical matrix-layer analogue''.  The DNN metric block is \(F_{\ell-1}F_{\ell-1}^{\top}\otimes I_{d_\ell}\) because the same incoming feature Gram is repeated over \(d_\ell\) output rows.  The generic RG secant Gram \(\Gamma_j^{\rm sec}\in\mathbb R^{p_j\times p_j}\) acts on the coupling vector \(g_j\).  Without additional channel structure it does not produce the Kronecker repetition that balances DNN local and atlas costs.

\paragraph{Free Gaussian RG as the fixed-kernel baseline}\label{subsec:gaussian-rg}

Let \(E=E_C\oplus E_F\) be finite-dimensional and let a centered Gaussian field have precision matrix
\[
Q=
\begin{pmatrix}
Q_{CC} & Q_{CF}\\
Q_{FC} & Q_{FF}
\end{pmatrix},
\qquad Q_{FF}\succ0.
\]
Define the Schur-complement precision
\[
Q_C^{\rm RG}:=Q_{CC}-Q_{CF}Q_{FF}^{-1}Q_{FC}.
\]
Then marginalizing the fluctuation variables gives \(\phi_C\sim N(0,(Q_C^{\rm RG})^{-1})\), and conditionally
\[
\phi_F=-Q_{FF}^{-1}Q_{FC}\phi_C+\zeta,
\qquad \zeta\sim N(0,Q_{FF}^{-1})
\]
independently of \(\phi_C\).  This is a fixed-kernel RG step: all pointwise complexity is determined by the Schur-complement covariance and the fluctuation covariance.  It is therefore analogous to a kernel, GP, or NTK model rather than to nonlinear feature learning.

\paragraph{Hierarchical phase/subspace prior.}

For a positive semidefinite matrix \(G\in\mathbb R^{p\times p}\), a Euclidean radius \(R\), and a resolution \(u>0\), set the effective rank and the effective subspace
\[
r_{\mathrm{eff}}(G,R,u):=\max\{k:\lambda_k(G)R^2\ge u^2\},
\qquad
V_{\mathrm{eff}}(G,R,u):=\operatorname{span}\{\text{top }r_{\mathrm{eff}}\text{ eigenvectors}\}
\]
with the convention \(\max\emptyset:=0\).
Define the effective dimension
\begin{equation}\label{eq:rg-deff}
d_{\mathrm{eff}}(G,R,u):=\frac12\sum_{k=1}^{r_{\mathrm{eff}}(G,R,u)}
\log\!\left(\frac{8 R^2\lambda_k(G)}{u^2}\right).
\end{equation}
The following elementary estimate is the local-chart calculation used throughout.  For subspaces
\(V,\bar V\subset\mathbb R^p\), write
\[
\rho_{\mathrm{proj},G}(V,\bar V):=\|G^{1/2}(P_V-P_{\bar V})\|_{\mathrm{op}} .
\]

\begin{lemma}[Ellipsoidal local chart]\label{lem:ellipsoid-local-chart-rg}
Let \(r=r_{\mathrm{eff}}(G,R,u)\), let \(V=V_{\mathrm{eff}}(G,R,u)\), and let \(\bar V\in\operatorname{Gr}(p,r)\) satisfy
\[
\rho_{\mathrm{proj},G}(V,\bar V)\le \frac{u}{4R}.
\]
Let \(\pi_{\bar V}\) be the uniform law on \(B_2(2R)\cap\bar V\). Then uniformly over \(g\in B_2(R)\),
\[
\log\frac1{\pi_{\bar V}\{g':(g'-g)^\top G(g'-g)\le  u^2\}}
\le \frac12\sum_{k=1}^{r_{\mathrm{eff}}(G,R,u)}
\log\!\left(\frac{40 R^2\lambda_k(G)}{u^2}\right)\leq 2\,d_{\mathrm{eff}}(G,R,u).
\]
Equivalently, replacing \(u\) by a constant multiple gives the displayed local-chart bound used in Theorem~\ref{thm:rg-pointwise-dimension}.
\end{lemma}

Lemma~\ref{lem:ellipsoid-local-chart-rg} is based on the same finite-dimensional volume-ratio argument used in the DNN local-chart bound of \citet[Lemma 2]{li-xu-pointwise-generalization-dnn}; we refer to \citet[Section C.3]{li-xu-pointwise-generalization-dnn} for the proof. There are only two minor changes. First, for simplicity, we enlarge the global radius of the prior support from 1.58R to 2R, which does not affect the argument. Second, in the final step we use the cutoff definition~\eqref{eq:rg-deff} of the effective dimension to obtain
\begin{align*}
\frac12\sum_{k=1}^{r_{\mathrm{eff}}(G,R,u)}
\log\!\left(\frac{40R^2\lambda_k(G)}{u^2}\right)
=
d_{\mathrm{eff}}(G,R,u)
+\frac{\log 5}{2}\,r_{\mathrm{eff}}(G,R,u)
\le
2\,d_{\mathrm{eff}}(G,R,u).
\end{align*}
This observation also shows that the absolute-constant rescaling of the radius in effective-dimension terms of the form $d_{\mathrm{eff}}(G,C \cdot R,u)$ in \cite{li-xu-pointwise-generalization-dnn} can be refined to an absolute-constant multiplicative factor $C\cdot d_{\mathrm{eff}}(G,R,u)$, while keeping the effective-rank term at the original radius R. We leave the systematic incorporation of this refinement into the formal statements of \cite{li-xu-pointwise-generalization-dnn} and related follow-up work to future versions.

Let \(\mathcal A_j\) be a finite set of phase labels at level \(j\), with prior weights \(q_j(a)>0\).  Conditional on a phase label \(a\) and rank \(r\), let \(\nu_{j,a,r}\) be a prior on the Grassmannian \(\operatorname{Gr}(p_j,r)\).  For \(G_j=G_j(g,\varepsilon)\), define the atlas cost
\begin{equation}\label{eq:atlas-cost}
\mathfrak A_j(g,\varepsilon,u_j)
:=
\log\frac1{q_j(a_j(g))}
+
\log(1+p_j)
+
\log\frac1{\nu_{j,a_j,r_j}
\bigl(B_{\rho_{\mathrm{proj}},G_j}(V_j,u_j/(4R_j))\bigr)},
\end{equation}
where \(r_j=r_{\mathrm{eff}}(G_j,R_j,u_j)\), \(V_j=V_{\mathrm{eff}}(G_j,R_j,u_j)\), and \(a_j(g)\) is any chart containing the trajectory at level \(j\). 

\begin{theorem}[Finite-scale RG pointwise dimension]\label{thm:rg-pointwise-dimension}
Assume the metric domination \eqref{eq:rg-metric-domination}.  Choose resolutions \(u_j>0\) such that \(\sum_{j=0}^{J-1}u_j^2\le c\varepsilon^2\).  Construct a trajectory-independent hierarchical prior \(\Pi\) by independently, at each level \(j\), sampling a phase label \(a\), an effective rank \(r\), a reference subspace \(\bar V\in\operatorname{Gr}(p_j,r)\), and then sampling \(g_j\) uniformly in \(B_2(2R_j)\cap\bar V\).  Then, there exists an absolute constant $C>0$ such that uniformly for every trajectory \(g\),
\begin{equation}\label{eq:rg-pointwise-dim-main}
\log\frac1{\Pi(B_\rho(g,\varepsilon))}
\le
C\sum_{j=0}^{J-1}
\left[
 d_{\mathrm{eff}}(G_j(g,\varepsilon),R_j,u_j)
 +
 \mathfrak A_j(g,\varepsilon,u_j)
\right].
\end{equation}
\end{theorem}

\begin{proof}
For each scale choose the phase label \(a_j(g)\), the active rank \(r_j\), and a reference subspace \(\bar V_j\) within projection distance \(u_j/(4R_j)\) of \(V_j\).  Lemma~\ref{lem:ellipsoid-local-chart-rg} gives the local prior mass of the ellipsoidal \(u_j\)-ball inside this chart.  Multiplying the scale-wise prior masses and incorporating the phase, rank, and subspace probabilities yields the right-hand side of \eqref{eq:rg-pointwise-dim-main}. For a detailed explanation of the corresponding decomposition, we refer to \citet[Section~3.3.3]{li-xu-pointwise-generalization-dnn}.  Finally, if each scale perturbation lies in its ellipsoidal \(u_j\)-ball, then \eqref{eq:rg-metric-domination} and \(\sum_j u_j^2\le c\varepsilon^2\) imply \(g'\in B_\rho(g,\varepsilon)\).  Hence the product event is contained in the \(\rho\)-ball, and its prior mass lower-bounds \(\Pi(B_\rho(g,\varepsilon))\).
\end{proof}

The theorem is the RG counterpart of the DNN Riemannian-dimension calculation.  The local term \(d_{\mathrm{eff}}\) is the effective dimension of the finite-scale response Gram.  The atlas term is the price of making the prior independent of the realized RG trajectory.  Irrelevant directions do not contribute once their eigenvalues fall below the finite resolution; relevant and marginal directions are precisely the active directions that remain in the sum.

The theorem deliberately leaves $\mathfrak A_j$ explicit; structural mechanisms are needed to make the global-atlas cost comparable to the local-chart cost. In DNN, this balance comes from the matrix-weight/common-feature structure and the resulting Kronecker-repetition metric. In genuine RG settings, it must instead come from physical structure in the operator atlas, with the nonconvex example below providing a clean instance of cost balancing without a DNN-style matrix structure.

\subsection{A finite nonconvex phase-atlas example}\label{subsec:nonconvex-phase-model}

The previous theorem is abstract.  Here is the finite-volume picture to keep in mind.  A globally multimodal measure may be hard to control by one convex potential, but after one chooses a stable phase chart the fluctuation integration has a Schur-complement curvature.  The local cost is then the Gaussian or ellipsoidal cost inside that chart, while the global cost is only the code length of the phase/subspace atlas.

Let \(E=E_C\oplus E_F\), let \(\mathcal A\) be a finite set of phase labels, and write
\[
\nu(d\phi_C,d\phi_F)=\sum_{a\in\mathcal A}w_a\nu_a(d\phi_C,d\phi_F),
\qquad
\nu_a=Z_a^{-1}e^{-S_a(\phi_C,\phi_F)}
\mathbf 1_{\Omega_{C,a}\times\Omega_{F,a}}\,d\phi_Cd\phi_F,
\]
where \(w_a>0\), \(\sum_a w_a=1\), and the chart domains \(\Omega_{C,a}\) and \(\Omega_{F,a}\) are closed convex sets with nonempty interior.  Assume all integrals are finite and that, on each product chart,
\begin{equation}\label{eq:phase-chart-hessian-assumption}
\nabla^2S_a(\phi_C,\phi_F)\succeq
H_a=
\begin{pmatrix}A_a&B_a\\ B_a^\top&D_a\end{pmatrix},
\qquad
D_a\succ0,
\qquad
H_a^{\rm RG}:=A_a-B_aD_a^{-1}B_a^\top\succ0.
\end{equation}

\begin{theorem}[One-step nonconvex phase-atlas RG stability]\label{thm:phase-atlas-rg}
Under \eqref{eq:phase-chart-hessian-assumption}, the coarse marginal is again a finite phase-atlas mixture
\[
\nu_C(d\phi_C)=\sum_{a\in\mathcal A}w_a^+\nu_{C,a}(d\phi_C),
\qquad
\nu_{C,a}=Z_{C,a}^{-1}e^{-S_a^{\rm RG}(u)}\mathbf 1_{\Omega_{C,a}}\,du,
\]
where
\[
S_a^{\rm RG}(u):=-\log\int_{\Omega_{F,a}}e^{-S_a(u,v)}\,dv.
\]
Moreover \(S_a^{\rm RG}\) is \(H_a^{\rm RG}\)-strongly convex on \(\Omega_{C,a}\).  Hence every centered linear chart process
\[
Z_{a,x}:=\langle v_x,u\rangle-\mathbb E_{\nu_{C,a}}\langle v_x,u\rangle
\]
is sub-Gaussian with comparison metric
\begin{equation}\label{eq:phase-chart-canonical-metric}
d_a(x,y)^2=\langle v_x-v_y,(H_a^{\rm RG})^{-1}(v_x-v_y)\rangle .
\end{equation}
Thus the usual sub-Gaussian chaining and the pointwise peeling argument give chartwise pointwise envelopes.  To make the event simultaneous over phases, one allocates failure probabilities \(\delta_a\), which adds the atlas code \(\log(1/\delta_a)\), for instance \(\log(1/w_a^+)\) or \(\log |\mathcal A|\).
\end{theorem}

\begin{proof}
The marginal formula is just integration over \(E_F\).  If the normalized chart measures \(\nu_a\) are used, the normalized marginal weights are still \(w_a^+=w_a\); otherwise the weights are renormalized by the corresponding marginal partition functions.

The curvature statement is the only real step.  Subtract the coarse quadratic form from the Hessian lower bound.  The Schur complement identity gives
\[
H_a-
\begin{pmatrix}H_a^{\rm RG}&0\\0&0\end{pmatrix}
=
\begin{pmatrix}B_aD_a^{-1}B_a^\top&B_a\\ B_a^\top&D_a\end{pmatrix}
=
\begin{pmatrix}B_aD_a^{-1/2}\\D_a^{1/2}\end{pmatrix}
\begin{pmatrix}D_a^{-1/2}B_a^\top&D_a^{1/2}\end{pmatrix}
\succeq0 .
\]
Therefore
\[
(u,v)\longmapsto S_a(u,v)-\frac12u^\top H_a^{\rm RG}u
\]
is convex on the product chart.  Adding the convex indicators of \(\Omega_{C,a}\) and \(\Omega_{F,a}\) preserves convexity.  Pr\'ekopa--Leindler then implies that
\[
S_a^{\rm RG}(u)-\frac12u^\top H_a^{\rm RG}u
\quad\text{is convex on }\Omega_{C,a}.
\]
Equivalently, the coarse chart is \(H_a^{\rm RG}\)-strongly log-concave.

For a linear observable, the Herbst argument (the standard route from a log-Sobolev inequality to Gaussian concentration) applied to a strongly log-concave measure gives
\[
\mathbb E_{\nu_{C,a}}\exp\{\lambda(Z_{a,x}-Z_{a,y})\}
\le
\exp\left(\frac{\lambda^2}{2}
\langle v_x-v_y,(H_a^{\rm RG})^{-1}(v_x-v_y)\rangle\right),
\]
which is the metric \eqref{eq:phase-chart-canonical-metric}.  Applying the standard sub-Gaussian chaining bound on a fixed chart and then the same pointwise peeling used in Theorem~\ref{thm:pointwise-gaussian} gives the stated envelopes.  A final union bound over phase labels with probabilities \(\delta_a\) produces the atlas code.
\end{proof}

\begin{proposition}[Selected phase subfamily versus global class-wide cover]
\label{prop:phase-atlas-global-cover-separation}
In the coarse charts of Theorem~\ref{thm:phase-atlas-rg}, let \(\mathcal T_{\rm full}\) be a finite observable class and let \(\mathcal T_a\subseteq\mathcal T_{\rm full}\) be the subfamily relevant to phase \(a\).  Define
\[
\ell_a(x)^2:=\langle v_x,(H_a^{\rm RG})^{-1}v_x\rangle,
\qquad
L_a:=\sup_{x\in\mathcal T_a}\ell_a(x),
\qquad
L_{\rm glob}:=\sup_{b\in\mathcal A}\sup_{x\in\mathcal T_{\rm full}}\ell_b(x).
\]
For an atlas prior \(q\in\Delta(\mathcal A)\), with \(q(a)>0\), the selected phase validation is
\begin{equation}\label{eq:phase-atlas-selected-cover-bound}
\nu_{C,a}\!\left\{
\sup_{x\in\mathcal T_a}|Z_{a,x}|
\le
C L_a\sqrt{\log\frac{2|\mathcal T_a|}{\delta q(a)}}
\right\}
\ge 1-\delta q(a).
\end{equation}
If the phase and subatlas structure is ignored, the full class-wide singleton cover gives only
\begin{equation}\label{eq:phase-atlas-global-cover-bound}
\nu_{C,a}\!\left\{
\sup_{x\in\mathcal T_{\rm full}}|Z_{a,x}|
\le
C L_{\rm glob}\sqrt{\log\frac{2|\mathcal T_{\rm full}|}{\delta}}
\right\}
\ge 1-\delta .
\end{equation}
There are non-log-concave Gaussian-mixture phase atlases for which, on a selected phase with \(q(a)\) bounded below,
\[
\mathbb E\sup_{x\in\mathcal T_a}|Z_{a,x}|\asymp\sqrt{\log K},
\qquad
\mathbb E\sup_{x\in\mathcal T_{\rm full}}|Z_{a,x}|\asymp\sqrt{\log(NK)}.
\]
Hence the full-cover scale can be strictly larger.  If \(q\) is uniform over all \(N\) phases, the selected bound already pays \(\log N\); in that case this particular atlas advantage disappears.
\end{proposition}

\begin{proof}
The two displayed bounds are elementary consequences of the sub-Gaussian estimate from Theorem~\ref{thm:phase-atlas-rg}.  For each \(a,x\),
\[
\mathbb E_{\nu_{C,a}}e^{\lambda Z_{a,x}}
\le
\exp\left(\frac{\lambda^2\ell_a(x)^2}{2}\right),
\qquad \lambda\in\mathbb R.
\]
Thus
\[
\nu_{C,a}\{|Z_{a,x}|>t\}
\le 2\exp\left(-\frac{t^2}{2\ell_a(x)^2}\right),
\]
with the deterministic interpretation if \(\ell_a(x)=0\).  A union bound over \(\mathcal T_a\), with failure probability \(\delta q(a)\), gives \eqref{eq:phase-atlas-selected-cover-bound}; a union bound over \(\mathcal T_{\rm full}\), with the worst chart scale \(L_{\rm glob}\), gives \eqref{eq:phase-atlas-global-cover-bound}.  For continuous classes, the same comparison replaces singleton counts by local and global covering or chaining functionals.

We now give the promised strict example.  Fix \(N,K\ge2\).  Let
\(E_C=\mathbb R^{N(K+1)}\) with orthonormal coordinates
\(e_{a,0},e_{a,1},\ldots,e_{a,K}\), let \(E_F=\mathbb R\), and set
\[
S_a(u,v)=\frac12\|u-m_a\|_2^2+\frac12v^2,
\qquad
m_a=B e_{a,0},
\qquad
B\ge4\sqrt{\log(N+1)}.
\]
Take the product charts to be the whole spaces.  Then \(H_a^{\rm RG}=I\) and \(\nu_{C,a}=N(m_a,I)\).  With weights \(w_1=1/2\) and \(w_a=1/[2(N-1)]\) for \(a\ge2\), the coarse mixture is not log-concave.  Indeed, if \(f\) is its density and \(z=(m_1+m_2)/2\), then \(f(m_i)\ge w_i\varphi(0)\), while every component contributes at most \(\varphi(0)e^{-B^2/4}\) at \(z\).  Hence \(f(z)\le \varphi(0)e^{-B^2/4}<\sqrt{w_1w_2}\varphi(0)\), contradicting the midpoint inequality required by log-concavity.

Define
\[
\mathcal T_{\rm full}:=\{(b,k):1\le b\le N,
1\le k\le K\},
\qquad
Y_{b,k}(u)=\langle e_{b,k},u\rangle,
\qquad
\mathcal T_a:=\{(a,k):1\le k\le K\}.
\]
Under the selected chart \(a=1\), the centered variables
\[
Z_{1,(b,k)}=\langle e_{b,k},G\rangle,
\qquad G\sim N(0,I),
\]
are independent standard Gaussians over the full index set.  Therefore
\[
\mathbb E\max_{1\le k\le K}|G_{1,k}|\asymp\sqrt{\log K},
\qquad
\mathbb E\max_{1\le b\le N,\,1\le k\le K}|G_{b,k}|\asymp\sqrt{\log(NK)}.
\]
This proves the strict separation when \(q(1)\) is bounded below.  If \(q\) is uniform, \eqref{eq:phase-atlas-selected-cover-bound} contains \(\log(1/q(1))=\log N\), so the selected chart pays the same phase uncertainty as the full cover.  The gain is therefore a real prior/trajectory-aligned atlas gain, not a free removal of an unknown phase.
\end{proof}

\begin{corollary}[Pure-phase charts for a finite-volume double-well \(\Phi^4\) model]\label{cor:double-well-atlas-rg}
Let \(G=(V,E)\) be finite and
\[
S(\phi)=\frac12\sum_{\{u,v\}\in E}c_{uv}(\phi_u-\phi_v)^2+
\lambda\sum_{u\in V}(\phi_u^2-a^2)^2.
\]
Fix \(m>a/\sqrt3\) and define the two pure-phase charts
\[
\Omega_+=\{\phi:\phi_u\ge m\ \forall u\},
\qquad
\Omega_-=-\Omega_+.
\]
On either chart the Hessian is bounded below by
\begin{equation}\label{eq:double-well-chart-hessian}
L_G+4\lambda(3m^2-a^2)I,
\end{equation}
where \(L_G\) is the graph Dirichlet Hessian.  Consequently Theorem~\ref{thm:phase-atlas-rg} applies to every finite splitting \(V=C\sqcup F\) on these pure-phase restrictions.  A full double-well decomposition would require additional droplet or interface charts with their own Hessian bounds.
\end{corollary}

\begin{proof}
On \(\Omega_+\cup\Omega_-\), \(|\phi_u|\ge m\), and
\[
\frac{d^2}{dt^2}\lambda(t^2-a^2)^2
=4\lambda(3t^2-a^2)
\ge 4\lambda(3m^2-a^2)>0.
\]
Adding the graph Dirichlet Hessian gives \eqref{eq:double-well-chart-hessian}.  The positive diagonal curvature makes the full chart Hessian positive definite, so the fluctuation block and the Schur complement in any finite decomposition are positive.  Theorem~\ref{thm:phase-atlas-rg} then applies.
\end{proof}

\paragraph{Interpretation and relation to KLS.}
The finite nonconvex phase-atlas example should be read as a finite-cutoff bridge between convex concentration theory, high-dimensional geometry, and interacting-field intuition.  Inside one stable phase chart, a uniform Hessian lower bound gives Brascamp--Lieb covariance domination and Gaussian-type concentration for linear and Lipschitz observables \citep{brascamp-lieb-1976,bakry-emery-1985}.  A mixture of several such charts may be genuinely multimodal, so there need not be one global convex potential or one useful global Gaussian comparison.  The atlas separates the two effects: chartwise Schur-complement curvature gives the local Gaussian metric, while the phase or subspace prior pays the global nonconvexity.

This is not meant to answer the Kannan--Lovasz--Simonovits problem.  KLS asks for dimension-free isoperimetric and Poincare-type control for arbitrary isotropic log-concave measures \citep{kannan-lovasz-simonovits-1995}.  The finite-cutoff statement here is a different, more ``nonconvex'' question.  We assume stable charts where the cutoff effective action has explicit curvature, and we use that curvature to obtain an anisotropic Gaussian comparison on the chart.  Global non-log-concavity is not solved by a Cheeger inequality; it is recorded as an atlas cost.  Directions below the finite resolution are treated as approximation error rather than controlled through a global convex-geometric constant.

\section{Graph Local-Time Application }\label{sec:graph-background}

\subsection[Local time versus cover and blanket time]{Local time versus cover and blanket time}
 We give a specialized application of the pointwise-complexity principle to the control of local times\footnote{Local time is a standard tool towards cutoff-uniform local integrability control in constructive QFT \citep{dynkin-local-times-quantum-fields-1984}.} in Gaussian free fields. The results separate two distinct ways of using Gaussian information, highlighting the advantage of a field-level route. In the field-level route, one fixes a target set, controls the pinned GFF on that set by a pointwise envelope, and then transfers the resulting estimate through the Ray--Knight theorem to a local-time statement. In the global route, one first establishes that the entire graph has been covered and only afterwards deduces that the target set has been visited. These two stopping-time arguments may occur at different scales.

We prove this separation on phase-star hierarchies.  A target subatlas \(I\) is encoded by an atlas prior \(q\), and the same ambient pointwise envelope can be evaluated on the selected target.  The exact law
\[
\mathbb P\{\tau_I\le \tau(t)\}=(1-e^{-t})^{|I|},
\qquad
\mathbb P\{\tau_{\mathrm{cov}}\le \tau(t)\}=(1-e^{-t})^N
\]
shows that target cover has intrinsic inverse-root-local-time scale \(\log |I|\), whereas any route through full cover pays \(\log N\).  The pointwise envelope pays the target code length \(\log(|I|/q(I))\), and the chart-decorated phase-star adds a compressed within-chart term.

For readers coming from learning theory rather than probability, this is the graph version of the local/global separation.  The selected target set \(I\) plays the role of a pointwise subatlas; the full cover or blanket event plays the role of a global covering number over all phases.  The similarity between pointwise/local dimension, covering number, cover time, and blanket time is only partially developed here, but the star examples make the basic scale separation explicit.

\paragraph{Pinned GFF on graph, local time, and Ray--Knight.}

Let \(G=(V,E,c)\) be a finite connected weighted graph with symmetric conductances \(c_{xy}=c_{yx}\ge0\), where \(c_{xy}>0\) exactly when \(\{x,y\}\in E\).  Write
\[
c_x:=\sum_{y\in V}c_{xy}.
\]
The weighted graph Laplacian acts on functions \(f:V\to\mathbb R\) by
\[
(Lf)(x)=\sum_y c_{xy}\bigl(f(x)-f(y)\bigr).
\]
We consider the continuous-time random walk, started at a distinguished root \(v_0\), that jumps from \(x\) to \(y\) at rate \(c_{xy}/c_x\).  Its normalized local time is
\[
L_t(x)=\frac1{c_x}\int_0^t \mathbf 1\{X_s=x\}\,ds,
\qquad
\tau(t):=\inf\{s\ge0:L_s(v_0)>t\}.
\]

The pinned Gaussian free field with root \(v_0\) is the centered Gaussian vector \(\eta=\{\eta_x\}_{x\in V}\) such that \(\eta_{v_0}=0\) almost surely and whose covariance is the killed Green function
\[
\mathbb E[\eta_x\eta_y]
=
\Gamma_{v_0}(x,y)
:=\mathbb E_x[L_{T_{v_0}}(y)],
\]
where \(T_{v_0}\) is the hitting time of \(v_0\).  Equivalently, this covariance is the inverse of the pinned Laplacian on \(V\setminus\{v_0\}\).  Its canonical metric is
\[
d(x,y)=\bigl(\mathbb E|\eta_x-\eta_y|^2\bigr)^{1/2}=\sqrt{R_{\mathrm{eff}}(x,y)},
\]
where \(R_{\mathrm{eff}}\) is the effective-resistance metric of the conductance network.

We use the following finite-graph form of the second Ray--Knight theorem \citep{dynkin-local-times-quantum-fields-1984}.  Let \(\eta'\) be an independent copy of the pinned GFF.
\begin{proposition}[Second Ray--Knight isomorphism on a finite graph]\label{prop:ray-knight-graph}
With the above normalization, under the product law of the walk and the independent fields \((\eta,\eta')\),
\begin{equation}\label{eq:ray-knight}
\Bigl(L_{\tau(t)}(x)+\tfrac12\eta_x^2\Bigr)_{x\in V}
\;\stackrel{d}{=}\;
\Bigl(\tfrac12(\eta_x'+\sqrt{2t})^2\Bigr)_{x\in V}.
\end{equation}
In particular, \(L_{\tau(t)}(v_0)=t\) almost surely.
\end{proposition}
 \paragraph{Cover and blanket time.}
For a finite connected conductance graph \(G=(V,E,c)\), let \(\tau_{\mathrm{cov}}\) be the first time every vertex has been visited, and write
\begin{equation}\label{eq:tcov}
t_{\mathrm{cov}}(G):=\max_{v\in V}\mathbb E_v[\tau_{\mathrm{cov}}].
\end{equation}
A blanket time strengthens cover time by requiring that the walk has accumulated roughly stationary proportions of visits at all vertices. One convenient strong version leads to
\begin{equation}\label{eq:tbl}
t_{\mathrm{bl}}(G,\delta):=\max_{v\in V}\mathbb E_v[\tau_{\mathrm{bl}}(\delta)].
\end{equation}

\paragraph{Main results on local-time envelopes and cover/blanket criteria.}

Let \(H\subseteq V\) be a target set.  The proof below only requires a deterministic envelope for the pinned GFF on \(H\).  Such an envelope may be obtained either by applying Theorem~\ref{thm:pointwise-gaussian} to the restricted field on \(H\cup\{v_0\}\), or by applying it once to a larger ambient field and then evaluating the resulting bound on \(H\).  This distinction is useful in the phase-star examples, where one ambient atlas prior is fixed before the target subatlas is selected.

\begin{corollary}[Target-set pointwise local-time envelope]\label{cor:target-local-time-envelope}
Let \(A_H(\cdot;\delta)\) be any deterministic function satisfying
\[
\mathbb P\bigl\{ |\eta_x|\le A_H(x;\delta)\text{ for all }x\in H\bigr\}\ge 1-\delta
\]
for the pinned GFF.  Then, under the law of the walk, with probability at least \(1-\delta\),
\begin{equation}\label{eq:target-local-time-envelope}
\forall x\in H:\qquad
\bigl[t-\sqrt{2t}\,A_H(x;\delta/2)\bigr]_+
\le L_{\tau(t)}(x)
\le t+\sqrt{2t}\,A_H(x;\delta/2)+\frac12A_H(x;\delta/2)^2.
\end{equation}
\end{corollary}

Let
\[
\tau_H:=\inf\{s\ge0: L_s(x)>0\text{ for every }x\in H\}
\]
be the target-cover time of \(H\).  The preceding envelope gives the following target-cover and target-blanket criterion.
\begin{corollary}[Target cover and target blanket criterion]\label{cor:target-cover-threshold}
Assume that for some \(t>0\), \(H\subseteq V\), and \(\theta\in(0,1/\sqrt2)\),
\[
\sup_{x\in H}A_H(x;\delta/2)\le \theta\sqrt t.
\]
Then, under the law of the walk, with probability at least \(1-\delta\),
\[
(1-\sqrt2\,\theta)t
\le L_{\tau(t)}(x)
\le \Bigl(1+\sqrt2\,\theta+\frac12\theta^2\Bigr)t
\qquad \forall x\in H.
\]
In particular, \(\tau_H\le\tau(t)\) on this event, and
\[
\max_{u,v\in H}\frac{L_{\tau(t)}(u)}{L_{\tau(t)}(v)}
\le
\frac{1+\sqrt2\,\theta+\theta^2/2}{1-\sqrt2\,\theta}.
\]
\end{corollary}

\begin{proof}[Proof of Corollary~\ref{cor:target-local-time-envelope}]
Set
\[
A_x:=A_H(x;\delta/2),\qquad
U_x:=L_{\tau(t)}(x)+\frac12\eta_x^2,
\qquad
V_x:=\frac12(\eta_x'+\sqrt{2t})^2,
\qquad x\in H.
\]
By Theorem~\ref{prop:ray-knight-graph}, \((U_x)_{x\in H}\stackrel d=(V_x)_{x\in H}\).  Define
\[
D_x:=\frac12(\sqrt{2t}-A_x)_+^2,
\qquad
B_x:=t+\sqrt{2t}\,A_x+\frac12A_x^2.
\]
By the envelope assumption applied to the independent copy \(\eta'\), the event \(\{|\eta_x'|\le A_x\text{ for all }x\in H\}\) has probability at least \(1-\delta/2\).  On this event, \(D_x\le V_x\le B_x\) for all \(x\in H\).  Equality in distribution therefore gives
\[
\mathbb P\{D_x\le U_x\le B_x\text{ for all }x\in H\}\ge1-\delta/2.
\]
Applying the same envelope to \(\eta\) itself and taking a union bound, with probability at least \(1-\delta\), we also have \(|\eta_x|\le A_x\) for all \(x\in H\).  On the intersection of these events,
\[
L_{\tau(t)}(x)=U_x-\frac12\eta_x^2\le B_x,
\]
and
\[
L_{\tau(t)}(x)\ge D_x-\frac12A_x^2.
\]
Using nonnegativity of local time and the identity
\[
\max\left\{0,\frac12(\sqrt{2t}-A_x)_+^2-\frac12A_x^2\right\}
=\bigl[t-\sqrt{2t}\,A_x\bigr]_+
\]
gives \eqref{eq:target-local-time-envelope}.  The event in the conclusion depends only on the walk, so the same probability bound holds under the walk law.
\end{proof}

\begin{proof}[Proof of Corollary~\ref{cor:target-cover-threshold}]
Under the displayed small-envelope assumption, Corollary~\ref{cor:target-local-time-envelope} gives, simultaneously for all \(x\in H\),
\[
L_{\tau(t)}(x)\ge (1-\sqrt2\,\theta)t>0
\]
and
\[
L_{\tau(t)}(x)\le \Bigl(1+\sqrt2\,\theta+\frac12\theta^2\Bigr)t.
\]
Thus every vertex in \(H\) has been visited by time \(\tau(t)\), and the displayed ratio bound follows by dividing the upper estimate by the lower estimate.
\end{proof}

\subsection{Separating examples: target subatlas versus full cover}

We now give two finite graph models.  The first is a star, where the only geometry is the choice of target leaves.  The second decorates every phase by a dense local chart, so that the within-chart effective resistance is small.  In both cases the message is the same: an ambient pointwise envelope can be fixed before the target is selected and then evaluated on the selected subatlas.  A proof that first establishes full cover loses this target dependence and pays for all phases.

\begin{theorem}[Phase-star target subatlas versus full cover]\label{thm:phase-star-separation}
Let \(S_N\) be the unit-conductance star with root \(v_0\) and leaves \(\{1,\ldots,N\}\).  Let \(\mathcal I\) be a family of nonempty target sets and let \(q\in\Delta(\mathcal I)\).  Define
\[
H_I:=I,
\qquad
\tau_I:=\inf\{s\ge0:L_s(i)>0\text{ for every }i\in I\},
\]
and the ambient prior
\begin{equation}\label{eq:star-atlas-prior}
\mu_q:=\frac12\delta_{v_0}
+\frac14\mathrm{Unif}\{1,\ldots,N\}
+\frac14\sum_{I\in\mathcal I}q(I)\,\mathrm{Unif}(I).
\end{equation}
Then the pinned GFF satisfies \(\eta_1,\ldots,\eta_N\) independent \(N(0,1)\), and the pointwise Gaussian envelope with prior \(\mu_q\) gives, with probability at least \(1-\delta\), simultaneously for all \(I\in\mathcal I\),
\begin{equation}\label{eq:star-atlas-envelope}
\sup_{i\in I}|\eta_i|
\le
C\left[
\sqrt{\log\frac{4|I|}{q(I)}}+
\sqrt{\log\!\left(\frac{2+\log\log(N+3)}{\delta}\right)}
\right].
\end{equation}
Consequently, for every \(\theta\in(0,1/\sqrt2)\), if
\begin{equation}\label{eq:star-atlas-target-upper}
t\ge C_\theta\left[
\log\frac{4|I|}{q(I)}+
\log\!\left(\frac{2+\log\log(N+3)}{\delta}\right)
\right],
\end{equation}
then \(\mathbb P\{\tau_I\le\tau(t)\}\ge1-\delta\).  On the other hand,
\begin{equation}\label{eq:star-exact-coupon-rigorous}
\mathbb P\{\tau_I\le\tau(t)\}=(1-e^{-t})^{|I|},
\qquad
\mathbb P\{\tau_{\rm cov}\le\tau(t)\}=(1-e^{-t})^N.
\end{equation}
Thus target cover has inverse-root-local-time quantile \(\log |I|+O_\delta(1)\), while full cover has quantile \(\log N+O_\delta(1)\).  If \(\log(|I|/q(I))=o(\log N)\), the direct target route is asymptotically smaller than any route that validates the target event only through full cover.
\end{theorem}

\begin{proof}
The pinned Laplacian of a star, after removing the root, is the identity matrix, so the pinned GFF on the leaves is an independent standard Gaussian vector.  For a leaf \(i\), \(\sigma(i)=d(i,v_0)=1\), and \(d(i,j)=\sqrt2\) for distinct leaves.  If \(i\in I\), then the selected part of \(\mu_q\) gives
\[
\mu_q(\{i\})\ge \frac{q(I)}{4|I|}.
\]
Since the leaf diameter is an absolute constant,
\[
\Phi_{\mu_q}(i)
\le C\sqrt{\log\frac{4|I|}{q(I)}}
\qquad (i\in I).
\]
The uniform component gives \(\mu_q(\{i\})\ge1/(4N)\) for every leaf, hence \(\Phi_{\mu_q,*}\le C\sqrt{\log(N+1)}\).  With \(r_0=s_0=1\), the peeling multiplicity in Theorem~\ref{thm:pointwise-gaussian} contributes only the displayed \(\log\log N\) term.  This proves \eqref{eq:star-atlas-envelope}.  Corollary~\ref{cor:target-cover-threshold} then gives \eqref{eq:star-atlas-target-upper}, after increasing \(C_\theta\).

It remains to compute the exact cover laws.  During root visits, the walk waits an exponential time of rate one before jumping, and the next leaf is uniform on \(\{1,\ldots,N\}\).  Since \(c_{v_0}=N\), normalized root local time \(t\) corresponds to ordinary root occupation time \(Nt\).  Hence, before \(\tau(t)\), root departures form a Poisson process of mean \(Nt\) with independent uniform leaf marks.  By thinning, the number of excursions to each fixed leaf is Poisson with mean \(t\), independently across leaves.  A leaf has positive local time by \(\tau(t)\) exactly when its count is nonzero, giving \eqref{eq:star-exact-coupon-rigorous}.  Solving \((1-e^{-t})^k\ge1-\delta\) gives
\[
q_k(\delta):=-\log\bigl(1-(1-\delta)^{1/k}\bigr)=\log k+O_\delta(1),
\]
and the stated separation follows.
\end{proof}

The collapsed star has no within-phase geometry.  The decorated version below adds a local chart to each phase.  The dense clique plays the same role as a compressed local chart in DNN or RG: once the phase has been selected, the within-chart resistance is small, so the pointwise envelope pays only a compressed within-chart term.

\begin{theorem}[Decorated phase-star subatlas theorem]\label{thm:decorated-phase-star}
Fix \(N,K\ge1\) and \(\kappa>0\) with \(\kappa(K+1)\ge1\).  Let \(G_{N,K,\kappa}\) have root \(v_0\), phase centers \(c_1,\ldots,c_N\), and local chart vertices \(z_{a,1},\ldots,z_{a,K}\).  The root is connected to each \(c_a\) by conductance one, and
\[
C_a:=\{c_a,z_{a,1},\ldots,z_{a,K}\}
\]
is a complete graph with edge conductance \(\kappa\).  For \(I\in\mathcal I\), set \(H_{I,K}:=\bigcup_{a\in I}C_a\), and define
\begin{equation}\label{eq:decorated-atlas-prior}
\mu_{q,K}:=\frac12\delta_{v_0}
+\frac14\mathrm{Unif}\!\left(\bigcup_{a=1}^N C_a\right)
+\frac14\sum_{I\in\mathcal I}q(I)\,\mathrm{Unif}\!\left(\bigcup_{a\in I}C_a\right).
\end{equation}
Put
\begin{equation}\label{eq:decorated-envelope-scale}
\mathcal E_{I,K}(\delta):=
\sqrt{\log\frac{4|I|}{q(I)}}+
\sqrt{\frac{1}{\kappa(K+1)}\log\frac{4|I|(K+1)}{q(I)}}+
\sqrt{\log\!\left(\frac{2+\log\log((N+3)(K+3))}{\delta}\right)}.
\end{equation}
Then the full-graph pointwise envelope with prior \(\mu_{q,K}\) satisfies, with probability at least \(1-\delta\), simultaneously for all \(I\in\mathcal I\),
\begin{equation}\label{eq:decorated-envelope}
\sup_{x\in H_{I,K}}|\eta_x|\le C\mathcal E_{I,K}(\delta).
\end{equation}
Consequently, for every \(\theta\in(0,1/\sqrt2)\),
\begin{equation}\label{eq:decorated-target-cover}
t\ge C_\theta \mathcal E_{I,K}(\delta)^2
\quad\Longrightarrow\quad
\mathbb P\{\tau_{H_{I,K}}\le\tau(t)\}\ge1-\delta.
\end{equation}
Finally,
\begin{equation}\label{eq:decorated-center-coupon}
\mathbb P\{\{c_1,\ldots,c_N\}\subseteq X[0,\tau(t)]\}=(1-e^{-t})^N,
\end{equation}
so full cover at confidence \(1-\delta\) must pay inverse-root-local-time \(\log N+O_\delta(1)\).  In particular, if \(\mathcal E_{I,K}(\delta)^2=o(\log N)\), the selected target-cover route is asymptotically smaller than the full-cover route.
\end{theorem}

\paragraph{How to read the decorated phase-star.}
The centers \(c_a\) are global phase labels, while each clique \(C_a\) is a local chart.  The first term in \(\mathcal E_{I,K}\) is the target-atlas code.  The second term is the compressed within-chart cost; it is small when \(\kappa(K+1)\) is large.  This is the graph analogue of compressed feature directions in DNN or irrelevant directions in finite-cutoff RG.  A full-cover proof ignores the selected atlas and pays the coupon-collector cost of all \(N\) phase centers.

\begin{proof}
The resistance estimates explain the scale \(\mathcal E_{I,K}\).  In a complete graph with \(K+1\) vertices and edge conductance \(\kappa\), the resistance between two vertices is \(2/[\kappa(K+1)]\).  Adding the root can only reduce within-clique resistance, hence
\[
R_{\rm eff}(x,y)\le \frac{2}{\kappa(K+1)},
\qquad x,y\in C_a.
\]
The only connection from the root to the phase \(C_a\) is the edge \((v_0,c_a)\), so \(R_{\rm eff}(v_0,c_a)=1\), and the triangle inequality gives
\[
R_{\rm eff}(v_0,x)\le 1+\frac{2}{\kappa(K+1)},
\qquad x\in C_a.
\]
Thus \(\sigma(x)\) is bounded by an absolute constant.

Let
\[
\alpha_{K,\kappa}:=\sqrt{\frac{2}{\kappa(K+1)}}.
\]
For \(x\in C_a\) with \(a\in I\), the ball \(B_d(x,\alpha_{K,\kappa})\) contains all of \(C_a\), so \(\mu_{q,K}(B_d(x,\alpha_{K,\kappa}))\ge q(I)/(4|I|)\).  At smaller radii the singleton mass is at least \(q(I)/[4|I|(K+1)]\).  Splitting the localized integral defining \(\Phi_{\mu_{q,K}}(x)\) at \(\alpha_{K,\kappa}\) gives
\[
\Phi_{\mu_{q,K}}(x)
\le
C\sqrt{\log\frac{4|I|}{q(I)}}
+C\alpha_{K,\kappa}\sqrt{\log\frac{4|I|(K+1)}{q(I)}}.
\]
The uniform component of \(\mu_{q,K}\) gives singleton mass at least \(1/[4N(K+1)]\), and therefore \(\Phi_{\mu_{q,K},*}\le C\sqrt{\log((N+1)(K+1))}\).  Theorem~\ref{thm:pointwise-gaussian}, with \(r_0=s_0=1\), yields \eqref{eq:decorated-envelope}.  Corollary~\ref{cor:target-cover-threshold} then gives \eqref{eq:decorated-target-cover}.

For the full-cover obstruction, observe that every departure from the root chooses one of the phase centers uniformly.  The same Poisson-thinning computation as in Theorem~\ref{thm:phase-star-separation} gives \eqref{eq:decorated-center-coupon}.  Since full cover implies that every phase center has been visited,
\[
\mathbb P\{\tau_{\rm cov}\le\tau(t)\}
\le (1-e^{-t})^N.
\]
The \((1-\delta)\)-quantile of this coupon-collector event is \(\log N+O_\delta(1)\), completing the separation.
\end{proof}

\paragraph{Relation to classical cover-time theory.}
Let \(H(u,v)\) be the expected hitting time from \(u\) to \(v\), and define the commute-time metric
\[
\kappa(u,v):=H(u,v)+H(v,u).
\]
Let \(\mathcal C:=\sum_{x\in V}c_x\) be the total conductance.  The commute-time identity gives
\[
\kappa(u,v)=\mathcal C\,R_{\rm eff}(u,v).
\]
For an unweighted graph, \(\mathcal C=2|E|\).  If \(\eta=\{\eta_v\}_{v\in V}\) denotes the pinned GFF, then \(\mathbb E(\eta_u-\eta_v)^2=R_{\rm eff}(u,v)\).

\begin{proposition}[Ding--Lee--Peres \citep{ding-lee-peres-cover-2012}]\label{prop:tcov-gamma2}
For any connected conductance graph \(G=(V,E,c)\) and any fixed \(0<\delta<1\),
\begin{equation}\label{eq:tcov-gamma2}
t_{\rm cov}(G)\asymp \gamma_2\!\bigl(V,\sqrt{\kappa}\bigr)^2
=\mathcal C\cdot\gamma_2\!\bigl(V,\sqrt{R_{\rm eff}}\bigr)^2
\asymp_\delta t_{\rm bl}(G,\delta).
\end{equation}
Equivalently,
\begin{equation}\label{eq:tcov-gff}
t_{\rm cov}(G)\asymp \mathcal C\cdot\left(\mathbb E\max_{v\in V}\eta_v\right)^2.
\end{equation}
\end{proposition}

The point of the phase-star examples is not to reprove the Ding--Lee--Peres theorem.  That theorem answers the global question: how long does it take to cover every vertex of an arbitrary graph, and how is that time characterized by majorizing measures?  Here we ask a different, more local question: if a target subatlas is selected, can one control local time on that target before proving full cover?  Theorems~\ref{thm:phase-star-separation} and~\ref{thm:decorated-phase-star} show that the answer can be yes, and that the difference can be strict.

If the target set were fixed in advance and one were allowed to redo the Gaussian analysis only on that set, classical generic chaining on the restricted field would recover the target scale.  The pointwise formulation is useful for a different reason: one ambient prior and one full-field envelope are chosen first, and then the bound is evaluated on whichever target subatlas is selected.  This keeps the dependence on \(\log(|I|/q(I))\) and, in the decorated star, on the compressed within-chart term \((\kappa(K+1))^{-1}\log(|I|(K+1)/q(I))\).  A route through full cover collapses this information and pays the coupon-collector scale \(\log N\).

\section{Technical variants and Sudakov-type consequences}\label{app:sudakov}
In this section, we derive a Sudakov-type comparison from the Bayesian algorithmic lower bound for the nearest-neighbor decoder via a supremum relaxation. Our aim is conceptual: to show that the global Sudakov scale may be viewed as a coarse relaxation of a more localized algorithmic validation. We use this route only to clarify the soundness and intuition of the approach. Recent concurrent work \citep{zadik2026bayesian} shows that a comparison between nearest-neighbor and Bayes-optimal risks yields a succinct proof of the majorizing-measure lower bound. Together with the standard upper-bound arguments, this gives a concise  perspective on both sides of Talagrand's majorizing-measure theorem.

For a prior \(\pi\) and radius \(\Delta\), write
\[
q_\pi(\Delta):=\sup_{a\in T}\pi(B_d(a,\Delta)),
\qquad
L_\pi(\Delta):=\log\frac1{q_\pi(\Delta)},
\qquad
\mathsf V_\pi:=\iint d(s,t)^2\,\pi(ds)\pi(dt).
\]

\begin{corollary}[Sudakov-type comparison through supremum relaxation]\label{coro:subdakov}
In the Gaussian location setup of Theorem~\ref{thm:nn-comparison}, take the ordinary nearest-neighbor decoder and the centered reference law.  If \(L_\pi(\Delta)\) is above a universal constant, then choosing
\[
\tau^2\asymp \frac{\mathsf V_\pi}{L_\pi(\Delta)}
\]
yields
\begin{equation}\label{eq:variance-normalized-sudakov}
\mathbb E\sup_{s,u\in T}(X_s-X_u)
\gtrsim
\Delta^2\sqrt{\frac{L_\pi(\Delta)}{\mathsf V_\pi}}.
\end{equation}
In particular, if \(\mathsf V_\pi\le C_0\Delta^2\), then
\begin{equation}\label{eq:localized-sudakov-prior}
\mathbb E\sup_{s,u\in T}(X_s-X_u)
\gtrsim_{C_0}
\Delta\sqrt{L_\pi(\Delta)}.
\end{equation}
Consequently, with
\begin{equation}\label{eq:localized-frac-cover}
N_{\rm spread}^{\rm loc}(T,d,\Delta;C_0)
:=
\sup_{\pi:\,\mathsf V_\pi\le C_0\Delta^2}
\frac{1}{\sup_{a\in T}\pi(B_d(a,\Delta))},
\end{equation}
one has
\begin{equation}\label{eq:localized-sudakov-frac}
\mathbb E\sup_{s,u\in T}(X_s-X_u)
\gtrsim_{C_0}
\Delta\sqrt{\log N_{\rm spread}^{\rm loc}(T,d,\Delta;C_0)},
\end{equation}
whenever the logarithm is above the universal threshold required by the information condition.
\end{corollary}

\begin{proof}
For the centered reference choice \(m=\bar v_\pi\), Theorem~\ref{thm:nn-comparison} gives information radius
\[
\mathcal I_{\pi,m}=\frac{\mathsf V_\pi}{4\tau^2}.
\]
For nearest neighbor, the exact ghost mass is bounded by the class-wide small-ball mass:
\[
\rho_{\Delta,Q_{\tau,\bar v_\pi}}(\widehat\Theta_{\rm nn})
\le q_\pi(\Delta).
\]
Thus, if \(L_\pi(\Delta)\) is large enough, taking \(\tau^2\) to be a sufficiently large constant multiple of \(\mathsf V_\pi/L_\pi(\Delta)\) makes the information condition \eqref{eq:bayesian-nn-condition} hold with an absolute constant \(c\in(0,1)\).  The comparison lower bound \eqref{eq:nn-to-gaussian-range-main} then gives
\[
\mathbb E\bigl(X_{\widehat\Theta_{\rm nn}}-X_\Theta\bigr)
\gtrsim \frac{\Delta^2}{\tau}.
\]
Since \(X_{\widehat\Theta_{\rm nn}}-X_\Theta\le\sup_{s,u\in T}(X_s-X_u)\) pointwise, substitution of the chosen \(\tau\) gives \eqref{eq:variance-normalized-sudakov}.  The localized estimate \eqref{eq:localized-sudakov-prior} follows from \(\mathsf V_\pi\le C_0\Delta^2\), and optimizing over such priors gives \eqref{eq:localized-sudakov-frac}.
\end{proof}

The extra hypothesis \(\mathsf V_\pi\lesssim\Delta^2\) is not cosmetic.  It says that the hard prior used to create small balls is also spread at the same radius as the test.  Without this localization, the Gaussian channel must use a larger noise level to keep the information small, and the nearest-neighbor route loses exactly the factor visible in \eqref{eq:variance-normalized-sudakov}.

\begin{corollary}[Localized packing form]\label{cor:localized-packing-sudakov}
Let \(S\subset T\) be \(2\Delta\)-separated and contained in \(B_d(t_0,C_0\Delta)\) for some \(t_0\in T\).  If \(|S|=M\) is larger than a universal constant, then
\[
\mathbb E\sup_{s,u\in T}(X_s-X_u)
\gtrsim_{C_0}
\Delta\sqrt{\log M}.
\]
\end{corollary}

\begin{proof}
Take \(\pi\) uniform on \(S\).  Since \(S\) is \(2\Delta\)-separated, every \(\Delta\)-ball contains at most one point of \(S\), so
\[
q_\pi(\Delta)\le M^{-1},
\qquad
L_\pi(\Delta)\ge\log M.
\]
The localization assumption gives
\[
\mathsf V_\pi
=\iint d(s,t)^2\,\pi(ds)\pi(dt)
\le 4C_0^2\Delta^2.
\]
Corollary~\ref{coro:subdakov} then yields the claim.
\end{proof}

For comparison with classical covering notation, we record the finite-set fractional-covering duality used in \citet{chen-foster-han-qian-rakhlin-xu-2024}.  Recall that fractional and classical covering numbers are equivalent up to an absolute-constant rescaling of the covering radius \citep{block-dagan-rakhlin-2021,chen-foster-han-qian-rakhlin-xu-2024}.

\begin{lemma}[Fractional-covering duality]\label{lem:fractional-cover-duality}
For a finite metric space \((T,d)\) and any \(\Delta>0\),
\begin{equation}\label{eq:fractional-cover-duality}
N_{\rm frac}(T,d,\Delta)
:=
\inf_{\mu\in\Delta(T)}\sup_{x\in T}\frac{1}{\mu(B_d(x,\Delta))}
=
\sup_{\pi\in\Delta(T)}\frac{1}{\sup_{a\in T}\pi(B_d(a,\Delta))}.
\end{equation}
\end{lemma}

\begin{proof}
Let
\[
r_*:=\sup_{\mu\in\Delta(T)}\inf_{x\in T}\mu(B_d(x,\Delta)).
\]
Then \(N_{\rm frac}(T,d,\Delta)=1/r_*\).  Since \(T\) is finite, von Neumann's minimax theorem gives
\[
\begin{aligned}
r_*
&=\sup_{\mu\in\Delta(T)}\inf_{x\in T}
  \sum_{a\in T}\mu(a){\bf 1}\{a\in B_d(x,\Delta)\} \\
&=\inf_{\pi\in\Delta(T)}\sup_{\mu\in\Delta(T)}
  \sum_{x,a\in T}\pi(x)\mu(a){\bf 1}\{a\in B_d(x,\Delta)\} \\
&=\inf_{\pi\in\Delta(T)}\sup_{a\in T}\pi(B_d(a,\Delta)).
\end{aligned}
\]
Taking reciprocals proves \eqref{eq:fractional-cover-duality}.
\end{proof}

The equality \eqref{eq:fractional-cover-duality} optimizes the small-ball term, but it does not remove \(\mathsf V_\pi\) from \eqref{eq:variance-normalized-sudakov}.  That factor comes from the information radius of the Gaussian location channel.  Hence there is no contradiction between fractional-covering lower bounds and the localization requirement above: fractional-covering duality finds a hard prior for the entropy obstruction, while the nearest-neighbor Gaussian comparison also needs that hard prior to have second moment of order \(\Delta^2\).

The unrestricted classical Sudakov minoration is a separate theorem.  Standard proofs use Gaussian comparison and convex-geometric duality; see, for example, the expositions of \citet{pollard-sudakov-notes} and \citet[Chapter~3]{ledoux-talagrand-1991}. The lower half of the Fernique--Talagrand majorizing-measure theorem is stronger and multiscale \citep{ledoux-talagrand-1991,talagrand2005generic,vanhandel2018chainingii,zadik2026bayesian}.  The modest message here is that the Bayesian algorithmic lower envelope is a one-scale, decision-theoretic primitive.  After replacing the exact ghost mass by a class-wide small-ball mass and optimizing over localized priors, it becomes a Sudakov-type obstruction; the full unrestricted Sudakov and majorizing-measure lower bounds require the classical Gaussian machinery.

\section*{Acknowledgements}

The author used ChatGPT 5.5 Pro as a research and editorial assistant to refine preliminary proofs and improve the exposition. With the exception of Theorem~\ref{thm:penalty-range-information-relaxation}, all mathematical claims, physical interpretations, and final formulations were independently proved, verified, and revised by the author. Theorem~\ref{thm:penalty-range-information-relaxation}, including the algorithm-suggested penalty-range information-relaxation criterion and its relevance to broader questions of algorithmic robustness, was proposed by the author. Its proof was completed with the assistance of ChatGPT 5.5 Pro. The statistical formulation in terms of minimax analysis and high-dimensional asymptotics, as well as the proof, were subsequently checked and approved by the author. The author assumes full responsibility for the final form and validity of all results.

\end{document}